\documentclass{amsart}
\usepackage{verbatim}
\usepackage{latexsym}
\usepackage{amsmath}
\usepackage{amsfonts}
\usepackage{bbm}
\usepackage{float}
\usepackage{tikz}
\usepackage{placeins}
\usepackage{orcidlink}
\usetikzlibrary{positioning}
\newcommand{\countpct}[2]{
  #1\,{\scriptsize(#2\%)}
}
\usepackage{booktabs}
\usepackage{bbm}

\usepackage{tabularx,multirow,booktabs,makecell,amssymb}
\usepackage{siunitx}
\usepackage{amssymb,hyperref}
\usepackage[normalem]{ulem}

\usepackage[normalem]{ulem}
\usepackage{graphicx}
\usepackage{caption}
\usepackage{subcaption}

\renewcommand{\pmod}[1]{\hspace{0.05em}(\mathrm{mod}\,#1)}

\usepackage{dsfont}

\renewcommand{\pmod}[1]{\,{\rm mod}\,#1}

\newcommand{\C}{\mathbb{C}}
\newcommand{\F}{\mathbb{F}}

\newcommand{\Q}{\mathbb{Q}}
\newcommand{\Z}{\mathbb{Z}}

\newcommand{\dd}{\;\mathrm{d}}

\newtheorem{thm}{Theorem}[section]

\newtheorem{prop}[thm]{Proposition}

\newtheorem{lem}[thm]{Lemma}
\newtheorem{conj}[thm]{Conjecture}

\newtheorem{qu}[thm]{Question}

\theoremstyle{definition}
\newtheorem{rem}[thm]{Remark}

\numberwithin{equation}{section}
\numberwithin{table}{section}
\numberwithin{figure}{section}

\newcommand{\m}{\mathrm{m}}

\newcommand{\hypgeo}[2]{
  {\vphantom{F}}_{#1}\kern-\scriptspace F_{#2}
}

\usepackage{xcolor}

\newmuskip\pFqmuskip
\newcommand*\pFq[6][8]{
  \begingroup
  \pFqmuskip=#1mu\relax
  \mathchardef\normalcomma=\mathcode`,

  \mathcode`\,=\string"8000

  \begingroup\lccode`\~=`\,
  \lowercase{\endgroup\let~}\pFqcomma

  {}_{#2}F_{#3}{\left(\!\left.\genfrac..{0pt}{}{#4}{#5}\right|#6\!\right)}
  \endgroup
}
\newcommand{\pFqcomma}{{\normalcomma}\mskip\pFqmuskip}

\title{Machine learning the arithmetic of Boyd's Mahler measure conjectures}

\author{Alberto Alfarano}
\address{Alberto Alfarano: Axiom Math}
\email{alberto@axiommath.ai}

\author{Pablo Bianucci\orcidlink{0000-0002-1341-1437}}
\address{Pablo Bianucci: Department of Physics,
Concordia University, Montreal, QC H4B 1R6, Canada}
\email{pablo.bianucci@concordia.ca}

\author{Matilde N. Lal\'in\orcidlink{0000-0002-7155-5439}}
\address{Matilde N. Lal\'in:
D\'epartement de math\'ematiques et de statistique,
Universit\'e de Montr\'eal, Montreal, QC H3C 3J7, Canada}
\email{matilde.lalin@umontreal.ca}

\author{Berend Ringeling\orcidlink{0000-0003-2665-0063}}
\address{Berend Ringeling:
D\'epartement de math\'ematiques et de statistique,
Universit\'e de Montr\'eal, Montreal, QC H3C 3J7, Canada}
\email{bjringeling@gmail.com}


\begin{document}
\raggedbottom

\begin{abstract}
Boyd conjectured that the Mahler measure of $P_k(x,y)=x+y+\frac{1}{x}+\frac{1}{y}+k$ for $k$ an integer, is given by $r_kL'(E_k,0)$, where $E_k$ is the elliptic curve associated to the zero locus of $P_k$ and $r_k$ is a rational number. We study various arithmetic properties of $r_k$ using a dataset containing the first $250{,}000$ values of $k$, combining
large-scale statistical analysis assisted by \textsc{Claude} with transformer-based
experiments carried out using \textsc{Axolver}.

We recover Boyd's observation that, apart from a few exceptions, $r_k$ is
the reciprocal of an integer. The size of this integer is governed by the conductor of the elliptic curve. Moreover, its $p$-adic valuations display markedly different behavior according
to the prime. For $p\geq5$, the observed frequencies of $v_p(r_k)=-m$,
for $m\geq1$, are close to $p^{-m}$.
At $p=5$, the agreement improves substantially when
we restrict to curves with good reduction at $5$. For the primes  $2$ and $3$, however, we find additional arithmetic structure
involving congruence conditions on $k$ and the primes of bad reduction of
$E_k$. Although the neural networks do not predict $r_k$ exactly, they recover significant information about its magnitude and valuations.
In particular, the experiments at the prime $2$ suggest predictive
relationships within the computed range that are not fully captured
by the explicit predictor obtained from our statistical analysis.

\end{abstract}

\maketitle

The (logarithmic) Mahler measure of a nonzero rational function $P\in \C(x_1,\dots,x_n)$ is defined by
\begin{align*}
\mathrm{m}(P)=\frac{1}{(2\pi i)^n}\int_{\mathbb{T}^n} \log\left| P(x_1,\dots,x_n)\right|\frac{\dd x_1}{x_1}\cdots \frac{\dd x_n}{x_n},
\end{align*}
where $\mathbb{T}^n=\{(u_1,\dots,u_n)\in \C^n\, : \, |u_1|=\cdots=|u_n|=1\}$.
In the one variable polynomial case given by  $n=1$, Jensen's formula leads to an expression in terms of the roots of $P$. More precisely, if $P(x)= a \prod_j(x-\alpha_j)$, one can see that $\m(P)=\log |a| + \sum_j \log\max\{1,|\alpha_j|\}$.

However, for $n>1$, the nature of the values is more mysterious. In many cases where $P$ has rational coefficients, the Mahler measure has been related to values of functions with arithmetic significance.

Deninger~\cite{Deninger} predicted a connection of certain two-variable Mahler measures with special values of $L$ functions associated to elliptic curves via Be\u\i linson's conjectures, which are very general frameworks relating special values of $L$-functions to regulators, generalizing central statements such as the Birch and Swinnerton Dyer conjecture and the Dirichlet class number formula. These connections were verified by Boyd~\cite{Bo98} and Rodriguez-Villegas~\cite{RV} in many cases, most of them numerical, with a few formulas proven by Brunault, Mellit, Rogers and others \cite{Br06,La10,LR07,Me12,RZ12,RZ14,Zu14}.
One of the most studied families fitting the above framework is the one proposed by Boyd~\cite{Bo98}, given by
\[\mu_k:=\mathrm{m}\left(x+y+\frac{1}{x}+\frac{1}{y}+k\right),\]
where $k$ is an integral parameter.  Boyd computed $\mu_k$ for $k\in\Z\cap [0,40]$ and combined these observations with Deninger's predictions to conjecture that for  $k\in \Z\setminus\{-4,0,4\}$,
\begin{equation}\label{eq:conjecture}
\mu_k=r_k L'(E_k,0),
\end{equation}
where $E_k$ is the elliptic curve corresponding to the projective closure of the desingularization of
$x+y+\frac{1}{x}+\frac{1}{y}+k=0$, and $r_k$ is a nonzero rational number. In fact, apart from a few exceptions, it seems that $1/r_k$ is an integer. Table \ref{table-conj} shows $1/r_k$ for the first 40 values of $k$. The conductor $N_k$ is included in the table because it
enters the functional equation of the $L$-function and
governs the scale of $L'(E_k,0)$.

Moreover, the sign of $r_k$ should be determined by the root number $\omega_k$ of $E_k$, as the Mahler measure is positive. We remark, by using the change of variables $x\rightarrow -x$ and $y \rightarrow -y$ in the integral, that $\mu_k=\mu_{-k}$, and therefore it suffices to consider positive $k$.
When $k=0, 4$, the zero loci of the corresponding Laurent polynomials have genus 0. In these cases, $\mu_0=0$ and $\mu_4=2L'(\chi_{-4},-1)$, the Dirichlet series on the character of conductor 4 \cite{Bo98}.

\begin{table}[h]
\centering
\begin{tabular}{ccc||ccc}
\hline
$k$ & $1/r_k$ & $N_k$ & $k$ & $1/r_k$ & $N_k$ \\
\hline
1  & 1     & 15   & 21 & $-12$  & 1785  \\
2  & 1     & 24   & 22 & 24     & 3432  \\
3  & 1/2   & 21   & 23 & 6      & 1311  \\
4  & --    & --   & 24 & 8      & 840   \\
5  & 1/6   & 15   & 25 & $-16$  & 3045  \\
6  & 2     & 120  & 26 & 128    & 17160 \\
7  & 2     & 231  & 27 & 12     & 2139  \\
8  & 1/4   & 24   & 28 & $-2$   & 336   \\
9  & 2     & 195  & 29 & $-24$  & 4785  \\
10 & $-8$  & 840  & 30 & 240    & 26520 \\
11 & $-8$  & 1155 & 31 & $-16$  & 3255  \\
12 & 1/2   & 48   & 32 & 1/3    & 42    \\
13 & $-4$  & 663  & 33 & $-204$ & 35409 \\
14 & 8     & 840  & 34 & 256    & 38760 \\
15 & $-24$ & 3135 & 35 & 224    & 42315 \\
16 & 1/11  & 15   & 36 & 2      & 240   \\
17 & $-24$ & 4641 & 37 & $-208$ & 50061 \\
18 & $-16$ & 1848 & 38 & 288    & 54264 \\
19 & $-40$ & 6555 & 39 & 336    & 58695 \\
20 & 2     & 240  & 40 & $-8$   & 1320  \\
\hline
\end{tabular}
\caption{Numerical values of $1/r_k$ and the conductor of $E_k$.\label{table-conj}}
\end{table}

Some of these identities have been proven, see Table \ref{table:proven_integer}.
\begin{table}[h]
\centering
\begin{tabular}{cccc}
\hline
$k$ & $1/r_k$ & Conductor of $E_k$ & Reference \\
\hline
1  & 1      & 15   & Rogers and Zudilin \cite{RZ14} \\
2  & 1      & 24   & Rogers and Zudilin \cite{RZ12} \\
3  & 1/2    & 21   & Brunault \cite{Brunault-Siegel}; Lal\'in, Samart, and Zudilin \cite{LSZ16} \\
5  & 1/6    & 15   &  Lal\'in \cite{La10} from $k=1$ \\
8  & 1/4    & 24   & Lal\'in and Rogers \cite{LR07} from $k=2$\\
12 & 1/2    & 48   & Brunault \cite{Brunault-Siegel} \\
16 & 1/11   & 15   & Lal\'in \cite{La10} from $k=1$ \\
\hline
\end{tabular}
\caption{Proven formulas for integer values of $k$.}
\label{table:proven_integer}
\end{table}

Understanding the nature of these identities would shed light on our understanding of special values of $L$ functions and Be\u\i linson's conjectures. The following is a natural question to pose.
\begin{qu}
    How does $r_k$ relate to $k$ and to the elliptic curve $E_k$?
\end{qu}

We remark that $r_k$ should be predicted by the Bloch--Kato conjectures. Very recent work by Dummigan, Golyshev, de Jeu and Kerr \cite{DGJK} investigates the $2$-part of the Bloch--Kato conjecture for this family of curves.

We have computed the first $250{,}000$ values of $\mu_k$ and explored this question in two different ways: by analyzing the data with the assistance of  \textsc{Claude} and by training neural networks on the  problem via \textsc{Axolver}
, with the goal of predicting certain properties of $r_k$.

The data analysis
revealed several arithmetic regularities
exhibited by $r_k$. After accounting for the sign, which is governed by the root number, the reciprocal $n_k=1/r_k$ is almost always an integer. Its order of magnitude is explained by the conductor through the estimate
\[
        |n_k|\asymp \frac{N_k}{\log (k)},
\]
while its $p$-adic valuations show different behaviors
depending on the prime. For primes $p\geq5$, we compare
the observed positive valuations with the reference
probabilities $p^{-r}$ for $r\geq1$. This provides a useful description for several primes
over the computed range. At $p=5$, the agreement improves
substantially when we restrict to curves with good
reduction at $5$.

The primes $2$ and $3$, however, have a different behavior.
The $3$-adic valuation seems to be strongly influenced
by congruence conditions on $k$ and by the parity of the
number of bad primes congruent to $1\pmod3$.
For the $2$-adic valuation, the data suggest the lower bound
\[
v_2(n_k)\geq
\omega_{\mathrm{odd}}(k)
+2\,\omega_{\mathrm{odd}}(k^2-16)-c(k),
\qquad k\neq4,8,
\]
where $\omega_{\mathrm{odd}}(n)$ counts the distinct odd
prime divisors of $n$, and $c(k)$ is the correction term
defined in \eqref{eq:c-defn}.
The excess over this proposed lower bound also depends
on the $2$-adic class of $k$. In particular, the empirical
class means are approximately described by $2+s(k)$,
with $s(k)$ defined in \eqref{eq:s-defn}.
 The work of Dummigan, Golyshev, de Jeu and Kerr \cite{DGJK}
provides a formula for $v_2(n_k)$ for an infinite, positive-density
family of $k$, conditionally on the Bloch--Kato conjecture. The formula
depends on the 2-valuation of a Selmer group associated to an elliptic curve 2-isogenous to $E_k$

as well as on a regulator index
$\iota_{k/4,2}\in\{1,2\}$, which measures the index of an explicit
motivic cohomology class in the rank-one $\Z_{(2)}$-lattice determined
by the $2$-adic regulator. We briefly examine this connection in Section \ref{ssec:lowboundproof}.

The experiments showed us that \textsc{Axolver} is able to recover part of this arithmetic structure from the data, but that the full problem of predicting $r_k$ remains very difficult. While the model is unable to completely learn $r_k$, it  does learn meaningful information about its size, with the best behavior occurring when the conductor is included among the input features. In particular, the model begins to reproduce the observed distribution of $L(E_k,2)$. The valuation experiments are more mixed. For $p=5$ and $p=7$, the frequent zero class provides a strong baseline.
Some runs at $p=5$ correctly predict positive valuations, but their overall
validation accuracy at the final epoch is below that of the constant-zero
predictor. By contrast, the experiments for $v_2(n_k)$ and $v_3(n_k)$ show that the network can exploit arithmetic features such as the factorization of the discriminant or conductor, giving evidence that the statistical patterns observed in Section~\ref{sec:statistics} are visible to a learning model.
In the case of the prime $2$, the improved accuracy at the second
plateau suggests predictive relationships within the computed range
that are not fully captured by our explicit predictor $\hat{P}(k)$
(see Section~\ref{sec:secondplateau}).

This article is organized as follows.
Section \ref{sec:background} presents the general mathematical background for the problem. Section \ref{sec:data-generation} briefly discusses the generation of the data. The statistical observations are presented in Section \ref{sec:statistics}. Section \ref{sec:model} describes the model and the experiments run in \textsc{Axolver}, while  Section \ref{sec:machine} contains a detailed discussion on the results of those experiments. Section \ref{sec:conclusions} summarizes our findings and discusses directions for further research, while Section \ref{sec:appendix} contains an appendix with supplementary tables and figures.

\section*{Acknowledgments} This project was initiated during the workshop ``Using AI to find examples'' that took place in April 2026 at the Centre de recherches math\'ematiques. The authors are grateful to the organizers, particularly Andrew Granville, and to the CRM for the excellent working conditions. The authors are thankful to Fran\c cois Charton, Christopher Deninger, Jordan Ellenberg, Rob de Jeu, Dimitris Koukoulopoulos,  Shousen Lu, Riccardo Pengo, and Wadim Zudilin for several helpful discussions. Part of this work was conducted during the workshop ``AI and number theory" that took place in May 2026 at the American Institute of Mathematics. The authors thank the organizers and Anthropic for providing access and guidance to \textsc{Claude Max} during this event.

This work was partially supported by the Natural Sciences and Engineering Research Council of Canada (RGPIN-2019-06988 to PB, RGPIN-2022-03651 to ML), and the Fonds de recherche du Qu\'ebec - Nature et technologies (Projet de recherche en \'equipe 345672 to ML), the  Institut des sciences math\'ematiques (postdoctoral fellowship to BR),  the Centre de recherches math\'ematiques (postdoctoral fellowship to BR).

\section*{Declaration of Generative AI and AI-Assisted Technologies} During the preparation of this manuscript, the authors used artificial intelligence (AI) tools to assist with data analysis, machine learning modeling, and manuscript editing:

\begin{itemize}

\item  Data Analysis: Anthropic's \textsc{Claude}
model was used to assist with processing and
cleaning the dataset and to suggest arithmetic
patterns. These pattern searches were exploratory;
the resulting observations and conjectures are
distinguished from the proved statements.

\item Machine Learning: Axiom's \textsc{Axolver} framework was used to train sequence-to-sequence neural networks and execute mathematical/symbolic reasoning tasks within the study's machine learning pipeline.

\item Computational Checks and Editing: OpenAI's
\textsc{ChatGPT} was used to assist with computational
checks, analysis scripts, model evaluation, figure
preparation, and improvements to the readability,
grammar, and style of the author-drafted text.
\end{itemize}

The original research framework, data collection, and final interpretations were entirely conceptualized and directed by the human authors. After using these tools, the authors reviewed, verified, and edited all outputs to ensure accuracy and eliminate potential hallucinations. The authors accept full accountability for the content and integrity of the final work.

\section*{Data Availability Statement}

The code and datasets used in this study are publicly available
through Borealis, the Canadian Dataverse Repository
\cite{MahlerData}.

\section{Background and generalities} \label{sec:background}

The plane curve defined by the equation
\[x+\frac{1}{x} + y + \frac{1}{y} + k=0,\]
where $k$ is a parameter, is birationally equivalent to an elliptic curve (over $\Q(k)$). This can be seen by applying, for example, the birational change of variables

\[
x=\frac{8(k^2X+8Y)}{k(k^2-16X)},
\qquad
y=\frac{8(k^2X-8Y)}{k(k^2-16X)},
\]
with inverse
\[
X=-\frac{k^2xy}{16},
\qquad
Y=\frac{k^3(xy+1)(x-y)}{128},
\]
which
leads to the Weierstrass equation

\begin{equation}\label{eq:Weierstrass}
E_k: Y^2 = X\left(X^2 + \frac{k^2}{8}\left(\frac{k^2}{8} - 1\right)X + \frac{k^4}{16^2}\right).
\end{equation}
For fixed $k$, this curve is nonsingular provided that $k\neq 0,\pm4$.

If an elliptic curve $E$ is defined over $\Q$, one can construct its $L$-function as follows
\[L(E,s) = \prod_{\mathrm{good}\, p} (1-a_p p^{-s}+p^{1-2s})^{-1} \prod_{\mathrm{bad}\, p} (1-a_p p^{-s})^{-1}=\sum_{n=1}^\infty \frac{a_n}{n^s},\]
where for $p$ prime,
\[a_p = 1+p-\#E(\F_p)\]
denotes the trace of Frobenius.
Here the bad $p$ are the primes where the corresponding reduction of $E$ modulo $p$ is singular, and the good $p$ are all the other primes. The bad primes are exactly those that divide the conductor $N$ of the elliptic curve.

The Weierstrass model \eqref{eq:Weierstrass} is defined over $\Q(k^2)$. This observation prompted Rodriguez-Villegas \cite{RV} to conjecture that equation \eqref{eq:conjecture} extends to the case where $k^2$ is a positive integer and to prove the cases $k=2\sqrt{2}, 4\sqrt{2}$. A word of caution here is that, since $k \not \in \Q$, and any change of variables will be defined over $\Q(k)$, the choice of the model does make a difference. A few more cases where $k^2\in \Z$ have been proven, see Table \ref{table:proven_nonrational}.
\begin{table}[h]
\centering
\begin{tabular}{cccc}
\hline
$k$ & $1/r_k$ & Conductor of $E_k$ & Reference \\
\hline
$\sqrt{2}$     & 4  & 56  & Zudilin \cite{Zu14} \\
$2\sqrt{2}$   & 1    & 32  & Rodriguez-Villegas \cite{RV} \\
$3\sqrt{2}$   & 2/5    & 24  & Lal\'in \cite{La10} from $k=2$  \\
$4\sqrt{2}$    & 1    & 64  & Rodriguez-Villegas \cite{RV} \\
$i$            & 1/2  & 17  & Zudilin \cite{Zu14} \\
$2i$           & 1    & 40  & Mellit \cite{Mellit-Oberwolfach} \\
$3i$           & 1/5    & 15  &  Lal\'in \cite{La10} from $k=1$ \\
$i\sqrt{2}$   & 2/3    & 24  & Lal\'in \cite{La10} from $k=2$  \\
\hline
\end{tabular}
\caption{Proven formulas for quadratic irrational and non-real values of $k$.}
\label{table:proven_nonrational}
\end{table}

 In the present work, however, we restrict our main statistical and  machine-learning analysis to integral values of $k$. This is the setting of Boyd's original numerical experiments.
This restriction is also convenient computationally: generating data for $k=1,\ldots,2B$ appears slightly faster than carrying out the corresponding computations for all integral values $|m| \leq B$, where $m=k^2$ is no longer required to be a square. We briefly examine the nonsquare case in Section~\ref{sec:sqrtm}, where the near-integrality and valuation laws are found to persist in many cases.

The motivation for conjecturing formulas such as \eqref{eq:conjecture} comes from Be\u\i linson's conjectures, as explained in \cite{Bo98,Deninger,RV}.

\section{Numerical generation of data}
\label{sec:data-generation}

In order to generate the data needed to train the model, we need to compute both sides of equation \eqref{eq:conjecture}.
The computation of $L'(E_k,0)$ can be performed to very high precision in
\textsc{PARI/GP} \cite{PARI}. For the Mahler measure we have the following explicit formulas for real $k$
(see \cite[p.~2324]{RZ14}):
\begin{equation}
\label{HG}
\mu_k =
\begin{cases}
\log |k| - \dfrac{2}{k^2}\,
{}_4F_3\!\left(\frac{3}{2}, \frac{3}{2}, 1, 1;\, 2, 2, 2;\, \frac{16}{k^2}\right),
& \text{if } |k| \ge 4, \\[6pt]
\dfrac{|k|}{4}\,
{}_3F_2\!\left(\frac{1}{2}, \frac{1}{2}, \frac{1}{2};\, \frac{3}{2}, 1;\, \frac{k^2}{16}\right),
& \text{if } |k| < 4 .
\end{cases}
\end{equation}
Here ${}_4F_3$ and ${}_3F_2$ denote hypergeometric functions. This explicit
description allows one to compute $\mu_k$ to high precision for
$k \in \Z \setminus\{-4, 0, 4\}$.
 Both $\mu_k$ and $L'(E_k,0)$ were computed using
PARI/GP at approximately 19 digits of internal
precision. We then applied \texttt{lindep} to the
vector $(\mu_k,L'(E_k,0))$, without specifying an
additional precision parameter. If an integer
relation
\[
A\mu_k+B L'(E_k,0)\approx 0,
\qquad A\ne 0,
\]
was returned, we recorded the rational candidate
$r_k=-B/A$. No additional residual threshold or
bound on the coefficients was imposed. In our dataset,
the largest absolute numerator of $r_k$ is $11$,
and the largest denominator is
$40{,}167{,}578{,}001{,}408$. The numerator has absolute value $1$ in $249{,}993$ of the
$249{,}999$ cases, and the median denominator is
$299{,}186{,}683{,}904$.

To assess the reliability of this rational
reconstruction, we independently recomputed
$\mu_k$ and $L'(E_k,0)$ with precision targets of $50$ and $100$ decimal digits
(and $30$ additional guard digits in each case) for a reproducibly selected
random sample of 30 values of $k$ from
$1\leq k\leq 250{,}000$, $k\ne 4$.
More precisely, the computed relative discrepancy
\[
\frac{|\mu_k-r_kL'(E_k,0)|}{|\mu_k|}
\]
was below $2\times10^{-96}$ in the computations with requested
precision $50$ decimal digits and below $10^{-134}$ in those
with requested precision $100$ decimal digits.
The verification code and the selected values of $k$
are available in our code repository.

We computed the values $r_k$ for $1 \le k \le 250{,}000$ (with $k \neq 4$).
In total this required about \SI{9.6e7}{\second} ($\approx 3$ years) of
single-core-equivalent computation time, with the average computation time per value rising
from roughly \SI{35}{\second} for small $k$ to about \SI{16.5}{\minute} for
$k$ near 250{,}000. The computations were carried out on 21 CPU cores
across 17 servers with a combined \SI{2108}{\giga\byte} ($\approx
\SI{2.1}{\tera\byte}$) of memory.

\section{Statistical observations of the data}
\label{sec:statistics}

 We analyzed the rational values $r_k$ numerically
recognized by the procedure described in
Section~\ref{sec:data-generation}, for
$1\leq k\leq250{,}000$ with $k\neq4$.
Since $\mu_k>0$, the functional equation of the
$L$-function determines the sign of $r_k$ from the
root number $\omega_k$ of $E_k$.

We write $n_k:=1/r_k$, including the cases in which
$n_k$ is not an integer. In the computed dataset,
$n_k$ is an integer whenever
\begin{equation}\label{eq:intcond}
k\notin\{3,4,5,8,12,16,32\}.
\end{equation}
We examine the sizes and $p$-adic valuations of
these values $n_k$ below.

\subsection{The size of $|n_k|$}
We can use the approximation for $\mu_k$ and $L'(E_k,0)$ to find an estimate for $|n_k|$.
Using the hypergeometric description for the Mahler measure \eqref{HG}, we find
\begin{equation}
\label{mmk}
\mu_k = \log(k) + o(1)
\end{equation}
as $k \to \infty$. For $L'(E_k,0)$, we use the functional equation to arrive at
\begin{equation}
\label{lvk}
L'(E_k,0) = \omega_k \frac{N_k}{4 \pi^2 }L(E_k,2),
\end{equation}
where $N_k$ is the conductor of $E_k$ and $\omega_k$ the root number. Using the following lemma, we can estimate the size of $L(E_k,2)$.

\begin{lem}
\label{lem:lvbound}
For any elliptic curve $E$ over $\Q$, we have the bounds
    \[ (0.2117\ldots =) \, \frac{\zeta(3)^2}{\zeta(3/2)^2} \leq L(E,2) \leq \zeta(3/2)^2 \, (= 6.8245\ldots).\]
\end{lem}

\begin{proof}
We write the $L$-value in terms of its local factors
\[
L(E, 2) = \prod_{\text{primes }p} L_p(E,2),
\]
where
\[
L_p(E,2) = \begin{cases}(1-a_pp^{-2}+p\cdot p^{-4})^{-1},&p\nmid N\\ (1\pm p^{-2})^{-1}
,&p\|N\\ 1,& p^2|N,\end{cases}
\]
where $N$ is the conductor of $E$.
For good primes (i.e. $p \nmid N$), we have, by the Hasse-Weil bound, $|a_p| \leq 2 \sqrt{p}$, and therefore,
\[
(1 - p^{-3/2})^2 = 1 - 2p^{-3/2} + p^{-3} \leq 1-a_pp^{-2}+p\cdot p^{-4} \leq 1 + 2p^{-3/2} + p^{-3} = (1 + p^{-3/2})^2.
\]
Taking reciprocals,
\[
\frac{1}{(1 + p^{-3/2})^2} \leq L_p(E,2) \leq \frac{1}{(1 - p^{-3/2})^2}.
\]
For $p \| N$, the same estimates hold:
\[
(1-p^{-3/2})^2 \leq 1 - p^{-2} \leq 1 \pm p^{-2} \leq 1 + p^{-2} \leq (1 + p^{-3/2})^2,
\]
where the outer inequalities can be easily verified. Finally, for $p^2 | N$, it is immediate that these bounds hold. Taking the product over all primes $p$, we conclude that
\[
\frac{\zeta(3)^2}{\zeta(3/2)^2} = \prod_{\text{primes }p} \frac{1}{(1 + p^{-3/2})^2} \leq L(E,2) \leq \prod_{\text{primes }p} \frac{1}{(1 - p^{-3/2})^2} = \zeta(3/2)^2.
\]
\end{proof}

As a consequence of \eqref{mmk}, \eqref{lvk} and Lemma \ref{lem:lvbound}, we conclude that
\begin{equation}
\label{orderofnk}
|n_k| = \frac{|L'(E_k,0)|}{\mu_k} \asymp \frac{N_k}{\log(k)}.
\end{equation}
Thus, the order of magnitude of $n_k$ is completely determined by $k$ and the conductor $N_k$. In Figure~\ref{fig:distrL_E_2} we see the distribution of $L(E_k,2)$ in our dataset.
\begin{figure}[!htbp]
    \centering
   \includegraphics[scale = 0.5]{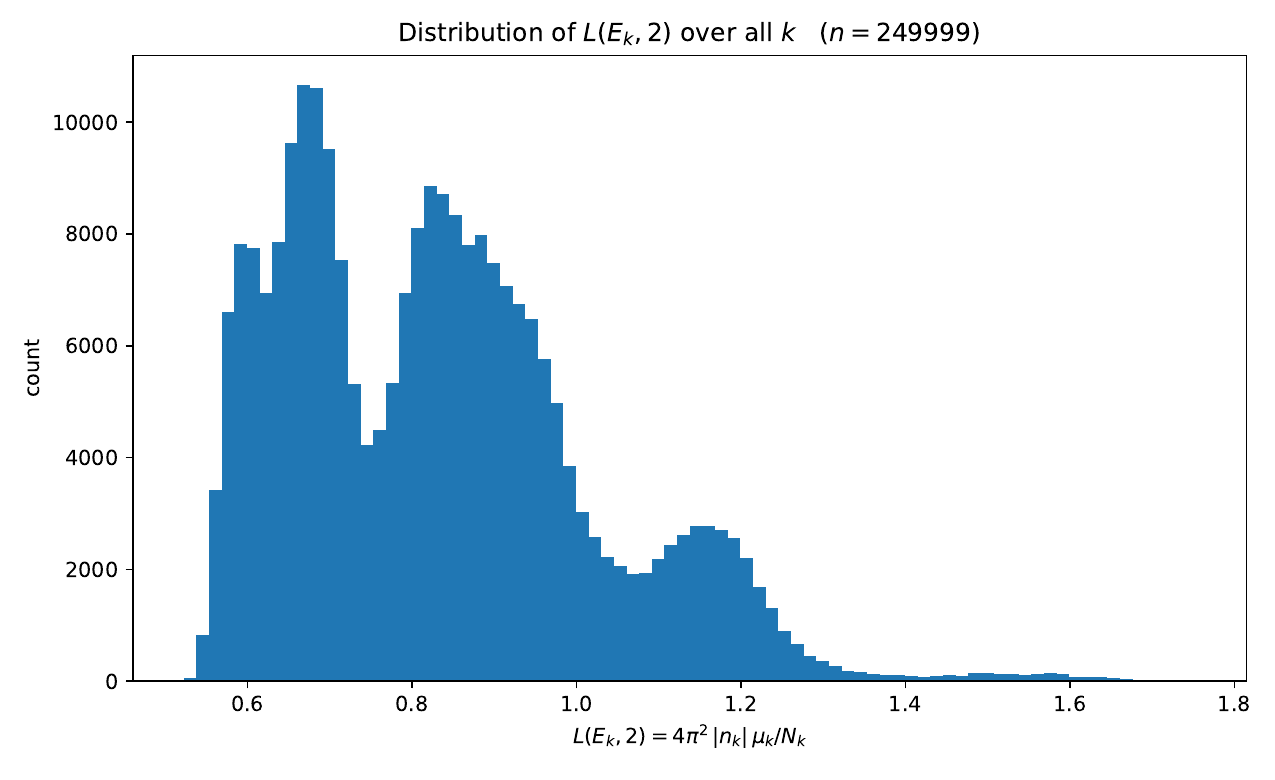}
    \caption{Values of $L(E_k, 2)$ for $k \leq 250{,}000$ in our dataset.}
    \label{fig:distrL_E_2}
\end{figure}

Although Lemma \ref{lem:lvbound} only forces $(0.2117,6.8245)$, the values we observe are in a narrower range $(0.5222, 1.7532)$, with mean value $0.8326$. The exact distribution of $L(E_k,2)$ is not known to us; here we record only its empirical behavior and its role in governing the size of $n_k$ through \eqref{orderofnk}. We do not assert any direct relation between this distribution and the $p$-adic valuations of $n_k$ considered below.

\subsection{$p$-adic observations for primes $p \geq 5$}
\label{subsection:big_p-adic}

Over the computed range, the observed frequencies suggest
the approximation
\[
\mathbb P\bigl(v_p(n_k)\geq1\bigr)\approx\frac1{p-1},
\]
compared with the probability $1/p$ predicted by a
uniform-distribution model modulo $p$.
Conditional on a positive valuation, the observed
frequencies are close to the usual geometric distribution:
\[
\mathbb P\bigl(v_p(n_k)=r\mid v_p(n_k)\geq1\bigr)
\approx
\left(1-\frac1p\right)p^{-(r-1)},
\qquad r\geq1.
\]
Together, these observations suggest the approximation
\[
\mathbb P\bigl(v_p(n_k)=r\bigr)\approx p^{-r},
\qquad r\geq1.
\]
The distinctive feature of the proposed model is therefore
the increased probability of the first factor of $p$;
the conditional distribution of the higher valuations
is the usual geometric distribution.
Experimental results for several primes $p$ and
valuations $r$ are recorded in
Table~\ref{tab:v5_v7_v11_v13_v17_v19_v23_v29_v31_v37_statistics_cumulative}.

To quantify how well our data matches with the proposed random model, we compute $\chi^2$. Concretely, for each prime $p$, we group the values $v_p(n_k)$ into finitely many bins. For instance, for $p = 5$, we use the seven bins $v_5(n_k) = r$ for $0 \leq r \leq 5$ and $v_5(n_k) \geq 6$. In these comparisons, the first bin is
$v_p(n_k)\leq0$.
For $p\geq5$, this differs from $v_p(n_k)=0$
only at $p=11$, $k=16$, where $n_{16}=1/11$.
The reference probability of this bin is
$(p-2)/(p-1)$. If $O_i$ is the observed number of values in bin $i$, and $E_i$ is the number predicted by the random model, then
\[
\chi^2=\sum_i \frac{(O_i-E_i)^2}{E_i}.
\]

We also report the normalized statistic $\chi^2/\nu$,
where
\[
\nu=\#\{\text{bins}\}-1.
\]
We use these statistics as descriptive measures of
agreement between the observed counts and the model.
Their values depend on the sample size, so small
differences in the observed frequencies can produce
elevated statistics. We interpret these comparisons
descriptively, without assuming that the observations
indexed by consecutive integers $k$ are independent
random samples.

In Table~\ref{tab:chi2-vs-prime}, we compare the observed
$p$-adic valuations with the model assigning probability
$p^{-r}$ to valuation $r\geq1$ and probability
$(p-2)/(p-1)$ to valuation $0$.
Throughout this section, $\mathbb P$ denotes the empirical
proportion obtained by sampling $k$ uniformly from the
indicated set of parameters.

 For each prime $5\leq p\leq47$, we make this comparison both for
the full dataset and for the subfamily for which $E_k$
has good reduction at $p$, equivalently
$p\nmid k(k^2-16)$.
 For each sample, we compute the chi-square statistic,
the number of degrees of freedom, and the normalized
statistic $\chi^2/\nu$. The expected counts are computed
using the size of the sample being compared.

For $p=5$, restricting to good reduction substantially
improves the agreement with the model: $\chi^2/\nu$
decreases from $12.81$ to $1.49$.
The counts for both samples are recorded in
Table~\ref{tab:v5-exact}.
For $p=37$, the corresponding normalized statistic
decreases from $4.76$ to $4.09$.

\begin{table}[H]
\centering

\begin{minipage}[t]{0.4\linewidth}
\vspace{0pt}
\centering
\scriptsize
\setlength{\tabcolsep}{3pt}
\begin{tabular}{r r r r r r r}
\toprule
& \multicolumn{3}{c}{all $k$}
& \multicolumn{3}{c}{$E_k$ good at $p$}\\
\cmidrule(lr){2-4}\cmidrule(lr){5-7}
$p$
& $\chi^2$ & $\nu$ & $\chi^2/\nu$
& $\chi^2$ & $\nu$ & $\chi^2/\nu$\\
\midrule
5  & 76.84 & 6 & 12.81 &  8.97 & 6 & 1.49\\
7  &  1.56 & 5 &  0.31 &  3.72 & 5 & 0.74\\
11 &  8.32 & 4 &  2.08 &  3.34 & 4 & 0.83\\
13 &  6.42 & 4 &  1.60 &  6.90 & 4 & 1.72\\
17 &  1.49 & 3 &  0.50 &  2.67 & 3 & 0.89\\
19 &  4.10 & 3 &  1.37 &  2.37 & 3 & 0.79\\
23 &  3.37 & 3 &  1.12 &  0.69 & 3 & 0.23\\
29 &  3.41 & 3 &  1.14 &  3.18 & 3 & 1.06\\
31 &  7.33 & 3 &  2.44 &  5.26 & 3 & 1.75\\
37 & 14.27 & 3 &  4.76 & 12.27 & 3 & 4.09\\
41 &  2.60 & 2 &  1.30 &  3.07 & 2 & 1.54\\
43 &  2.41 & 2 &  1.20 &  2.37 & 2 & 1.18\\
47 &  1.38 & 2 &  0.69 &  1.39 & 2 & 0.69\\
\bottomrule
\end{tabular}

\subcaption{Comparison with the model for primes
$5\leq p\leq47$, using all $k$ and, separately,
those $k$ for which $E_k$ has good reduction at $p$.}
\label{tab:chi2-vs-prime}
\end{minipage}
\hfill
\begin{minipage}[t]{0.53\linewidth}
\vspace{0pt}
\centering
\scriptsize
\setlength{\tabcolsep}{3pt}
\begin{tabular}{c r r c r r c}
\toprule
& \multicolumn{3}{c}{all $k$}
& \multicolumn{3}{c}{$E_k$ good at $5$}\\
\cmidrule(lr){2-4}\cmidrule(lr){5-7}
$v_5(n_k)=r$
& count & expected & ratio
& count & expected & ratio\\
\midrule
0 & $186{,}238$ & $187{,}499.3$ & $0.993$
  & $75{,}046$  & $75{,}000.0$ & $1.001$\\
1 &  $50{,}398$ &  $49{,}999.8$ & $1.008$
  & $19{,}802$  & $20{,}000.0$ & $0.990$\\
2 &  $10{,}651$ &  $10{,}000.0$ & $1.065$
  &  $4{,}101$  &  $4{,}000.0$ & $1.025$\\
3 &   $2{,}141$ &   $2{,}000.0$ & $1.071$
  &        $825$ &        $800.0$ & $1.031$\\
4 &         $471$ &         $400.0$ & $1.178$
  &        $182$ &        $160.0$ & $1.138$\\
5 &          $78$ &          $80.0$ & $0.975$
  &         $34$ &         $32.0$ & $1.063$\\
6 &          $19$ &          $16.0$ & $1.188$
  &          $9$ &          $6.4$ & $1.406$\\
7 &           $3$ &           $3.2$ & $0.938$
  &          $1$ &         $1.28$ & $0.781$\\
\bottomrule
\end{tabular}

\subcaption{Counts with $v_5(n_k)=r$ versus the expected
values. The left block uses all $M=249{,}999$ parameters,
while the right block uses the
$M_{\mathrm{good}}=100{,}000$ parameters for which $E_k$
has good reduction at $5$.
For $r\geq1$, the expected count is $M\cdot5^{-r}$
(respectively $M_{\mathrm{good}}\cdot5^{-r}$);
for $r=0$, the expected proportion is $\tfrac34$.}
\label{tab:v5-exact}
\end{minipage}

\caption{Comparison of the observed $p$-adic valuations
with the model assigning probability $p^{-r}$ to positive
valuation $r$. Left: the $\chi^2$ statistics for the full
dataset and the subfamily with good reduction at $p$.
Right: counts for $v_5(n_k)=r$ in the full dataset and
in the subfamily with good reduction at $5$.}
\label{tab:chi2-and-v5-exact}
\end{table}

\subsection{$3$-valuation observations}
The $3$-adic behavior of $n_k$ does not seem to follow the same pattern as the one described in Section~\ref{subsection:big_p-adic}. We record several observations, including parity laws for the $3$-valuations.

\subsubsection*{Globally higher $3$-divisibility.}

Divisibility by $3$ appears to occur considerably more frequently than the corresponding heuristic would suggest with $v_3(n_k)\geq 1$ about $73\%$ of the time; see  Table~\ref{tab:v3_statistics_cumulative} and Figure~\ref{fig:histogram_v3_0_250K}. More specifically, the proportions of $k$ for which $v_3(n_k)=r$ with $r \geq 2$ appear to be elevated (by at least a factor of $2$) compared with the baseline prediction $3^{-r}$. On the other hand, the proportion with $v_3(n_k)=1$ appears to be slightly smaller than this baseline: for instance, our data give approximately $0.322$, compared with $1/3 = 0.333\ldots$ expected from the model for $p\geq 5$. This excess is not an artifact of
the congruence phenomena discussed below.

\subsubsection*{$3$-Divisibility among congruence classes.}

The $3$-valuation is further structured by congruences. The proportion of $k$ with
$v_3(n_k)\geq r$ is systematically higher for $k\equiv 0\pmod 3$; see
Table~\ref{tab:v3_geq_mod3_statistics_congruence_cumulative}. Refining modulo $27$, we
find a stronger statement: for $k\equiv 0,\pm 5\pmod{27}$
one has $v_3(n_k)\geq 1$ for all but a few exceptional $k$ in the dataset.

\medskip
\noindent\emph{Bad primes of $E_k$ and $3$-divisibility.} Another observation is that there seems to be a correlation between the parity of the cardinality of the set
\[
S_k := \{ \text{bad primes of the form }p \equiv 1 \, \pmod 3 \text{ for the curve }E_k \}
\]
and the $3$-valuation of $n_k$, see Table \ref{tab:contingency_parity_S_T}. The table suggests that it is more likely for $n_k$ to be divisible by $3$ when the parity of $\#S_k$ is even: We have $\mathbb{P}(v_3(n_k) \geq 1 \, | \, \# S_k \text{ is even}) = 0.79$, compared to $\mathbb{P}(v_3(n_k) \geq 1 \, | \, \# S_k \text{ is odd}) = 0.67 $.
If we compare this to the analogous set
\[
T_k := \{ \text{bad primes of the form }p \equiv 2 \, \pmod 3 \text{ for the curve }E_k \},
\]
this effect disappears. We have $\mathbb{P}(v_3(n_k) \geq 1 \, | \, \# T_k \text{ is even}) = 0.73$, compared to $\mathbb{P}(v_3(n_k) \geq 1 \, | \, \# T_k \text{ is odd}) = 0.73$.
This phenomenon seems to be a consequence of a sharper law on a subfamily, which we will now describe.

\begin{table}[ht]
  \centering
  \small
  \setlength{\tabcolsep}{4pt}

  \begin{tabular}{lrrr}
    \toprule
    & $v_3(n_k) \leq 0$ & $v_3(n_k)\geq 1$ & total \\
    \midrule
    $\#S_k$ even
    & \countpct{$26{,}170$}{$20.9$}
    & \countpct{$98{,}896$}{$79.1$}
    & $125{,}066$ \\
    $\#S_k$ odd
    & \countpct{$40{,}811$}{$32.7$}
    & \countpct{$84{,}122$}{$67.3$}
    & $124{,}933$ \\
    \midrule
    total
    & \countpct{$66{,}981$}{$26.8$}
    & \countpct{$183{,}018$}{$73.2$}
    & $249{,}999$ \\
    \bottomrule
  \end{tabular}
\qquad
  \begin{tabular}{lrrr}
    \toprule
    & $v_3(n_k) \leq 0$ & $v_3(n_k)\geq 1$ & total \\
    \midrule
    $\#T_k$ even
    & \countpct{$33{,}406$}{$26.7$}
    & \countpct{$91{,}600$}{$73.3$}
    & $125{,}006$ \\
    $\#T_k$ odd
    & \countpct{$33{,}575$}{$26.9$}
    & \countpct{$91{,}418$}{$73.1$}
    & $124{,}993$ \\
    \midrule
    total
    & \countpct{$66{,}981$}{$26.8$}
    & \countpct{$183{,}018$}{$73.2$}
    & $249{,}999$ \\
    \bottomrule
  \end{tabular}

  \caption{Contingency tables comparing the divisibility of $n_k$ by $3$ with the parities of $\#S_k$ and $\#T_k$. Percentages are according to the parity of
$\#S_k$ and $\#T_k$.}
  \label{tab:contingency_parity_S_T}

\end{table}

\subsubsection{A parity law on a subfamily}
\label{ssec:partitylaw}

On the set
\begin{equation}
\label{def:setA}
A := \{ k \, \colon \, 3\nmid k, \text{ and } k(k^2-16)\text{ squarefree} \},
\end{equation}
the above-mentioned correlations sharpen into two implications
that hold throughout the computed range
$1\leq k\leq250{,}000$, with the exception of $k=1$:
\begin{equation}\label{eq:thelawforA}
v_3(n_k)=0 \Longrightarrow \#S_k \text{ is odd},
\qquad
v_3(n_k)=1 \Longrightarrow \#S_k \text{ is even}.
\end{equation}
see Tables \ref{tab:parity-law} and \ref{tab:law-contingency}. Here it can be seen that a similar law does not seem to appear for $v_3(n_k) \geq 2$.

\begin{table}[ht]
  \centering
  \begin{tabular}{rrrrr}
    \toprule
    $v_3(n_k)$ & count & $\#S_k$ even & $\#S_k$ odd & $\mathbb{P}(\#S_k\ \mathrm{odd})$ \\
    \midrule
    0 & $13{,}080$ & $1$ & $13{,}079$ & $1.000$ \\
    1 & $13{,}233$ & $13{,}233$ & $0$ & $0.000$ \\
    2 & $8{,}950$ & $4{,}438$ & $4{,}512$ & $0.504$ \\
    3 & $4{,}161$ & $2{,}055$ & $2{,}106$ & $0.506$ \\
    4 & $1{,}560$ & $780$ & $780$ & $0.500$ \\
    5 & $552$ & $266$ & $286$ & $0.518$ \\
    6 & $194$ & $95$ & $99$ & $0.510$ \\
    7 & $64$ & $31$ & $33$ & $0.516$ \\
    8 & $22$ & $10$ & $12$ & $0.545$ \\
    9 & $9$ & $5$ & $4$ & $0.444$ \\
    10 & $3$ & $1$ & $2$ & $0.667$ \\
    11 & $1$ & $0$ & $1$ & $1.000$ \\
    12 & $1$ & $1$ & $0$ & $0.000$ \\
    \bottomrule
  \end{tabular}
    \caption{Parity of $\#S_k$ as a function of $v_3(n_k)$, for $k \in A$. The law $v_3(n_k)=0\Rightarrow\#S_k$ odd and $v_3(n_k)=1\Rightarrow\#S_k$ even holds with a single exception; for $v_3(n_k)\ge2$ the parity is unconstrained.}
        \label{tab:parity-law}
\end{table}

\begin{table}[ht]
  \centering
  \begin{tabular}{lrrr}
    \toprule
               & $v_3(n_k) = 0 $ & $v_3(n_k) \geq 1 $ & total \\
    \midrule
    $\#S_k$ even  & $1$ & $20{,}915$ & $20{,}916$ \\
    $\#S_k$ odd   & $13{,}079$ & $7{,}835$ & $20{,}914$ \\
    \midrule
    total       & $13{,}080$ & $28{,}750$ & $41{,}830$ \\
    \bottomrule
  \end{tabular}
    \caption{Contingency of $\mathrm{parity}(\#S_k)$ against $v_3(n_k) \geq 1$ for $k \in  A$.}
  \label{tab:law-contingency}

\end{table}
 The observed laws \eqref{eq:thelawforA} suggest a
generalization to larger families. Define
\[
\alpha(k) := \#S_k + v_3(k) + \mathbf{1}_{8 \mid k}.
\]
The two laws are
\begin{subequations}\label{eq:law-general}
\begin{align}
\text{(odd)}&\qquad v_3(n_k)=0 \ \Longrightarrow\ \alpha(k)\ \text{is odd},
   \label{eq:law-fwd}\\
\text{(even)}&\qquad v_3(n_k)=1 \ \Longrightarrow\ \alpha(k)\ \text{is even}.
   \label{eq:law-conv}
\end{align}
\end{subequations}
On $A$ we have $k^2-16$ is squarefree, $v_3(k) = 0$ and $8 \nmid k$, so \eqref{eq:law-general} reduces to \eqref{eq:thelawforA}.
Table~\ref{tab:sets} and Figure~\ref{fig:placeholder}
compare several enlargements of $A$.
Among the families listed, the largest one with no
observed exceptions to the odd law~\eqref{eq:law-fwd}
in the range $5\leq k\leq250{,}000$ is
\[
C' := \{p^2 \nmid k^2 - 16 \text{ for all }p \equiv 3 \pmod 4 \} \setminus \{k \equiv 4 \pmod 8 \}.
\]
On this set, however, the even law \eqref{eq:law-conv} has $1{,}746$ exceptions. Of these, $1{,}597$ are divisible by $3^3$, and hence satisfy $v_3(n_k) \geq 1$ by the earlier remarks. For the remaining $149$, the integer $k(k^2-16)$ is divisible by one of $\{7^3, 13^3, 19^3, 37^3, 61^3 \}$, all cubes of primes congruent to $1 \pmod 3$. This suggests that both laws \eqref{eq:law-general} may hold on the smaller set
\[
C'' = \{k \in C' \, \colon\, p^3 \mid k(k^2-16) \Rightarrow p \equiv 2 \pmod 3 \}.
\]
The additional restriction defining $C''$ was motivated
by the exceptions to the even law on $C'$ within the
computed range. Thus the absence of exceptions to
either law on $C''$ in Table~\ref{tab:sets} is
exploratory evidence obtained from the same data
used to formulate this restriction.
Finally, it is worth noting that the $k$ such that $k \equiv 0 \pmod 27$ or $k \equiv \pm 5 \pmod 27$ lie outside the set $C''$:  the former because $3^3 \mid k$ and the latter because $9 \mid k^2-16$ obstructs squarefreeness at the prime  $3 \equiv 3 \pmod 4$. So their elevated $3$-valuation is not explained by the law on $C''$.
\begin{table}[h]\centering
\renewcommand{\arraystretch}{1.25}
\begin{tabular}{@{}lll cc@{}}
\toprule
set & definition & density & \eqref{eq:law-fwd} exc. & \eqref{eq:law-conv} exc.\\
\midrule
$A$      & $3\nmid k,\ k(k^2-16)$ squarefree      & $0.167$ & $0$ & $0$\\
$A'$     & $3\nmid k,\ k^2-16$ squarefree         & $0.184$ & $0$ & $46$\\
$A''$    & $k^2-16$ squarefree                    & $0.323$ & $0$ & $913$\\
$B$ & odd part of $k^2 - 16$ squarefree & $0.645$ & $3{,}076$ & $4{,}753$\\
$B'$ & $B \setminus \{ k \equiv 4 \pmod 8\}$ & $0.565$ & $0$ & $1{,}538$\\
$B''$ & $B' \cap \{p^3 \mid k\Rightarrow p\equiv 2 \pmod 3 \}$ & $0.536$ & 0 & 0  \\
$C$ & $p^2 \nmid k^2 - 16$ for all $p \equiv 3 \pmod 4$ & $0.721$ & $3{,}423$ &  $5{,}330$ \\
$C'$ & $C \setminus \{ k \equiv 4 \pmod 8 \}$ & $0.631$ & $0$ &  $1{,}746$ \\
$C''$ & $C' \cap  \{p^3 \mid k(k^2 - 16) \Rightarrow p\equiv 2 \pmod 3 \}$ & $0.598$ & $0$ & $0$ \\
 & all values $k$ & 1.000 & $11{,}342$ & $17{,}588$ \\
\bottomrule
\end{tabular}
\caption{The parity laws \eqref{eq:law-general} on different enlargements of $A$, computed over the range $5 \leq k\leq 250{,}000$. Here ``density'' is the proportion of integers $k$ that lie in the set, and ``exc.'' is the number of exceptions of the corresponding law.}
\label{tab:sets}
\end{table}
\begin{figure}[h]
  \centering
  \newcommand{\SCALE}{20}
  \begin{tikzpicture}[
      >=stealth,
      every node/.style={font=\normalsize},
      inc/.style={thin,->},
      dot/.style={circle,fill=black,inner sep=1.3pt}
  ]

  \foreach \d in {0.2,0.3,0.4,0.5,0.6,0.7}{

     \node[gray,font=\small,left] at (-1.0,\d*\SCALE) {\d};
  }
 \draw[->,gray] (-1.0,0.15*\SCALE) -- (-1.0,0.76*\SCALE) node[above,black,font=\small] {density};

  \node[dot,label=below:{$A$}]   (A)   at (0,    0.167*\SCALE) {};
  \node[dot,label=left:{$A'$}]   (Ap)  at (0,    0.184*\SCALE) {};
  \node[dot,label=left:{$A''$}]  (App) at (0,    0.323*\SCALE) {};
  \node[dot,label={[label distance=2pt]-30:{$B''$}}]  (Bpp) at (3.2,  0.536*\SCALE) {};
  \node[dot,label=left:{$B'$}]   (Bp)  at (3.2,  0.565*\SCALE) {};
  \node[dot,label=left:{$B$}]    (B)   at (3.2,  0.645*\SCALE) {};
  \node[dot,label=right:{$C''$}] (Cpp) at (6.4,  0.598*\SCALE) {};
  \node[dot,label=right:{$C'$}]  (Cp)  at (6.4,  0.631*\SCALE) {};
  \node[dot,label=right:{$C$}]   (C)   at (6.4,  0.721*\SCALE) {};

  \draw[inc] (A) -- (Ap);   \draw[inc] (Ap) -- (App);  \draw[inc] (A) -- (Bpp);
  \draw[inc] (App) -- (Bp); \draw[inc] (Bpp) -- (Bp);  \draw[inc] (Bp) -- (B);
  \draw[inc] (Cpp) -- (Cp); \draw[inc] (Cp) -- (C);    \draw[inc] (Bpp) -- (Cpp);
  \draw[inc] (Bp) -- (Cp);  \draw[inc] (B) -- (C);
  \end{tikzpicture}
  \caption{The sets of Table~\ref{tab:sets}, placed at height equal to their density.
  An arrow points from a set to a larger set containing it; since a subset is never
  denser than its superset, every arrow points upward.}
  \label{fig:placeholder}
\end{figure}

\subsection{$2$-adic observations}

Whenever $n_k$ is an integer it appears to be even, the only exceptions being
$k=1,2$ (where $n_k=1$). More is true: among integer values of $n_k$, the
inequality $v_2(n_k)\geq 2$ holds for all $k$ in the dataset outside the finite set
\[
k \in \{1,2,6,7,9,20,23,28,36,48,96,112,432\}.
\]
The distribution of $v_2(n_k)$ is recorded in
Table~\ref{tab:v2_statistics_cumulative} and
Figure~\ref{fig:histogram_v2_0_250K}.

\medskip
\noindent\emph{Dependence on congruence classes.} We next consider, for a fixed odd prime $q$, the proportion of $k$ in a residue
class $k\equiv a \pmod q$ with $v_2(n_k)\geq r$. For small $r$ this proportion
is close to $1$ in every class, so the structure only becomes visible for
larger $r$. Empirically, it is largest when $a\equiv\pm4\pmod q$ and second
largest when $a\equiv 0\pmod q$; see for example
Table~\ref{tab:v2_geq_mod7_statistics_congruence_cumulative}, where the
frequency of $v_2(n_k)\geq r$ is highest for $k\equiv\pm4\pmod 7$, closely
followed by $k\equiv0\pmod 7$. These are exactly the classes of bad reduction:
since $\Delta_k=2^{-24}k^{14}(k^2-16)$, we have
$\Delta_k\equiv0\pmod q$ if and only if $k\equiv 0,\pm4\pmod q$.

\medskip
\noindent\emph{Empirical model.} From the earlier observations, the primes dividing $k^2-16$ (the case
$k\equiv\pm4\pmod q$) have more influence than those dividing $k$ (the case
$k\equiv0\pmod q$), and empirically the appropriate weights are $2$ and $1$
respectively. We therefore assign to each odd prime $q$ the weight
\[
  W_q(k) := \begin{cases}
0 & q\nmid k(k^2-16),\\
1 & q\mid k,\\
2 & q\mid k^2-16,
\end{cases}\in\{0,1,2\},
\]
well defined since an odd prime divides at most one of $k,\,k-4,\,k+4$. Writing $\omega_{\mathrm{odd}}(n)$ for the number of
distinct odd prime divisors of $n$, we have
\[\sum_{q\ \mathrm{odd}} W_q(k)=\omega_{\mathrm{odd}}(k)+2\,\omega_{\mathrm{odd}}(k^2-16)=\omega_{\mathrm{odd}}(k)+2\,(\omega_{\mathrm{odd}}(k-4)+\omega_{\mathrm{odd}}(k+4)),\]
and find a moderately strong positive correlation (Pearson $r=0.69$) between this
quantity and $v_2(n_k)$. Since this sum involves only the odd primes, it appears to be natural to correct it by a term $c(k)$ that incorporates the local behavior of $E_k$ at the prime $2$. Let
\begin{equation}\label{eq:c-defn}
  c(k) =
  \begin{cases}
    4, & E_k \text{ semistable at } 2\ \ (v_2(N_k)\in\{0,1\}),\ \text{i.e. } k \text{ odd or } 16\mid k,\\
    2, & E_k \text{ additive at } 2\ \ (v_2(N_k)\in\{3,4\}),\ \text{otherwise.}
  \end{cases}
\end{equation}
We then set
\begin{equation}\label{eq:hatvdef}
  \hat v(k) := \sum_{q\ \mathrm{odd}} W_q(k) - c(k)
  = \omega_{\mathrm{odd}}(k) + 2\,\omega_{\mathrm{odd}}(k^2-16) - c(k).
\end{equation}
Then
$v_2(n_k)$ and $\hat v(k)$ have Pearson correlation $r=0.77$
(Figure~\ref{fig:correlationv2}), improving on the value $r=0.69$ obtained
without the correction $c(k)$. In fact, $\hat{v}(k)$ seems to underestimate $v_2(n_k)$ for
all but finitely many $k$.
\begin{figure}[ht]
  \centering
\includegraphics[width=0.85\textwidth]{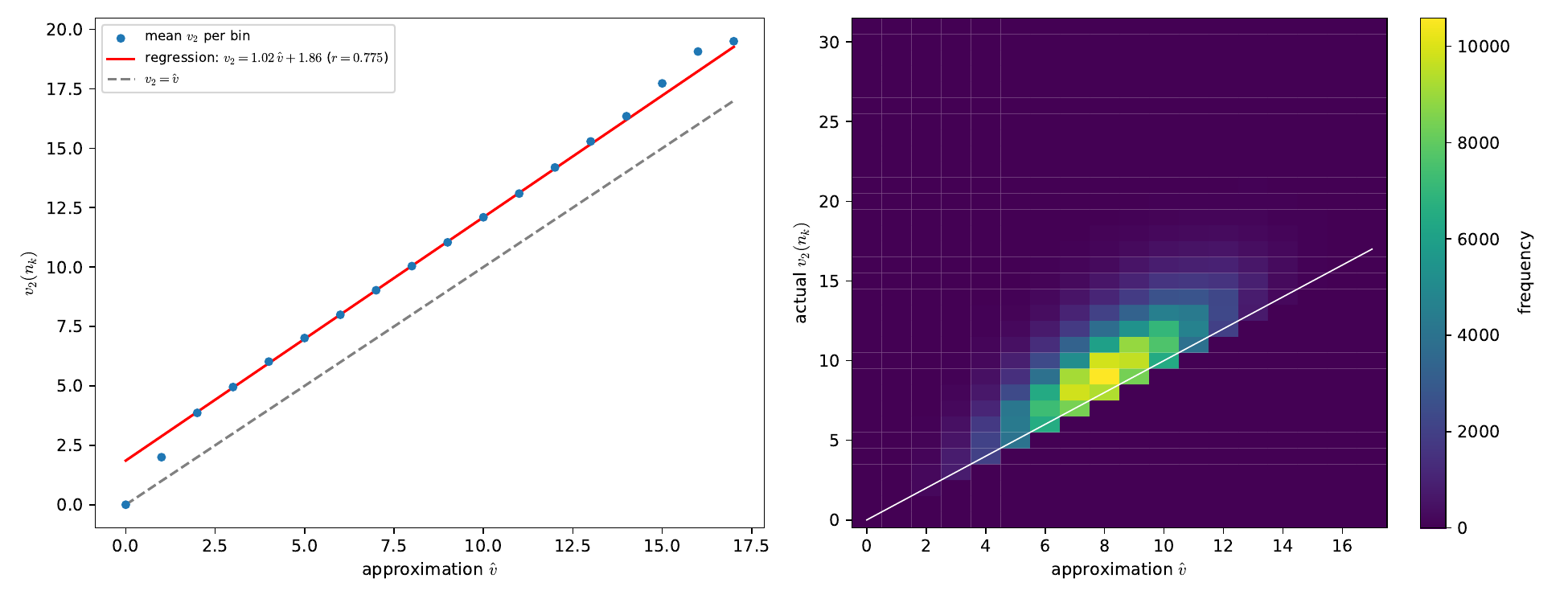}
  \caption{ $v_2(n_k)$ vs
    $\hat{v}(k)$ for $k$ with $v_2(n_k) \geq 0$}
  \label{fig:correlationv2}
\end{figure}
This leads us to the following conjecture.
\begin{conj}

\label{conj:low_bound}
For all $k \neq 4, 8$, we have
    \[\hat{v}(k) \leq v_2(n_k).\]
\end{conj}

Moreover, the excess $G(k) := v_2(n_k) - \hat{v}(k)$ has mean $2.05 $ and the data suggest that $G(k)=0$ may occur infinitely often, making the conjecture sharp in that sense.
This excess seems to spike whenever $v_2(k) = 3$, and then drops down below $2$ when $v_2(k) \geq 4$ (see Table \ref{tab:v2-left}). This spike comes from $k \equiv \pm 8 \pmod{64}$, where the mean excess is $3.02$ (see Table \ref{tab:v2-right}). In fact it grows further as $k \equiv \pm 8 \pmod{2^m}$ for increasing $m \geq 6$, plateauing at an excess of $3.5$ (see Table \ref{tab:v2-bottom}).
\begin{table}[h]

\centering
\begin{minipage}[t]{0.52\linewidth}
\centering
\begin{tabular}{l r r c}
\toprule
class & count & mean $v_2(n_k)$ & mean $G(k)$\\
\midrule
$k$ odd        & 125{,}000 & 9.56  & $+2.05$\\
$v_2(k)=1$     & 62{,}500  & 11.28 & $+2.04$\\
$v_2(k)=2$     & 31{,}249  & 10.39 & $+2.00$\\
$v_2(k)=3$     & 15{,}625  & 11.29 & $+2.41$\\
$v_2(k)=4$     & 7{,}813   & 8.67  & $+1.86$\\
$v_2(k)=5$     & 3{,}906   & 8.65  & $+1.91$\\
$v_2(k)=6$     & 1{,}953   & 8.55  & $+1.90$\\
$v_2(k)=7$     & 977    & 8.52  & $+1.93$\\
$v_2(k)\ge 8$  & 976    & 8.33  & $+1.95$\\
\midrule
all            & 249{,}999 & 10.15 & $+2.05$\\
\bottomrule
\end{tabular}
\subcaption{$v_2(n_k)$ by the $2$-adic class of $k$.}
\label{tab:v2-left}
\end{minipage}
\hspace{0.0015\linewidth}
\begin{minipage}[t]{0.42\linewidth}
\centering
\begin{tabular}{c r c c}
\toprule
$k\pmod 64$ & count & mean $v_2(n_k)$ & mean $G(k)$\\
\midrule
$8$  & $3{,}907$ & $11.96$ & $+3.09$\\
$24$ & $3{,}906$ & $10.69$ & $+1.81$\\
$40$ & $3{,}906$ & $10.70$ & $+1.82$\\
$56$ & $3{,}906$ & $11.83$ & $+2.94$\\
\midrule
all $v_2(k)=3$ & $15{,}625$ & $11.29$ & $+2.41$\\
\bottomrule
\end{tabular}
\subcaption{The class $v_2(k)=3$, refined by $k \pmod {64}$.}
\label{tab:v2-right}
\end{minipage}

\vspace{1em}

\begin{minipage}[t]{0.62\linewidth}
\centering
\begin{tabular}{c r r r r}
\toprule
$M$ & count & mean $v_2(n_k)$ & mean $G(k)$ & $2+s(k)$\\
\midrule
$5$        & $7{,}812$ & $10.69$ & $+1.81$ & $2.00$\\
$6$        & $3{,}906$ & $11.53$ & $+2.65$ & $2.75$\\
$7$        & $1{,}954$ & $12.19$ & $+3.31$ & $3.13$\\
$8$        & $976$  & $12.30$ & $+3.43$ & $3.31$\\
$9$        & $488$  & $12.39$ & $+3.50$ & $3.41$\\
$\ge 10$   & $489$  & $12.34$ & $+3.45$ & $\to 3.5$\\
\bottomrule
\end{tabular}
\subcaption{Excess by exact $2$-adic depth
$M:=\max\{v_2(k-8),v_2(k+8)\}$, compared with the
approximation $2+s(k)$ to the observed average excess.
The arrow in the last column indicates that $2+s(k)$
tends to $3.5$ as $M\to\infty$.}
\label{tab:v2-bottom}
\end{minipage}

\caption{$2$-adic valuation $v_2(n_k)$, with excess
$G(k)=v_2(n_k)-\hat v(k)$. Left: grouped by the $2$-adic class of $k$.
Right: the exceptional class $v_2(k)=3$, split further modulo $64$. Bottom: the exceptional classes by $2$-adic depth $M=\max\{v_2(k-8),v_2(k+8)$ \}. }
\label{tab:v2-valuation-classes}
\end{table}

A simple approximation to the empirical mean of $G(k)$
within the corresponding $2$-adic classes is $2+s(k)$,
where
\begin{equation}\label{eq:s-defn}
s(k):=\tfrac{3}{2}\,
\max\!\Big\{
1-2^{\,5-\max\{v_2(k-8),\,v_2(k+8)\}},\,0
\Big\}.
\end{equation}
Here $s(k)$ vanishes when $k$ is not $2$-adically close
to $\pm8$ and approaches $1.5$ as $k$ approaches $\pm8$
$2$-adically.

For comparison with the overall valuation histogram,
we use the integer-valued reference model
\begin{equation}\label{eq:v2approx}
\hat v(k)+\mathrm{round}\bigl(s(k)\bigr)+X_k,
\qquad
X_k\sim\mathrm{NB}\!\left(3,\tfrac35\right),
\end{equation}
where $\mathrm{round}$ denotes rounding to the nearest
integer and $X_k$ is an auxiliary random variable.
Here $\mathrm{NB}$ denotes the negative-binomial
distribution, with probability mass function
\[
\mathbb P(X_k=j)
=
\binom{j+2}{2}
\left(\frac35\right)^3
\left(\frac25\right)^j,
\qquad j=0,1,2,\ldots.
\]
This distribution has mean $2$ and mode $1$.
The mean excess in the rounded reference model is
therefore $2+\mathrm{round}(s(k))$. The quantity $2+s(k)$ approximates the observed averages
of $G(k)$ listed in Table~\ref{tab:v2-bottom}.

Figure~\ref{fig:histogramv2stat} compares this reference
model with the observed distribution of $v_2(n_k)$.
We regard it as a first-order description of the
overall histogram.

\begin{figure}
    \centering
    \includegraphics[scale=0.5]{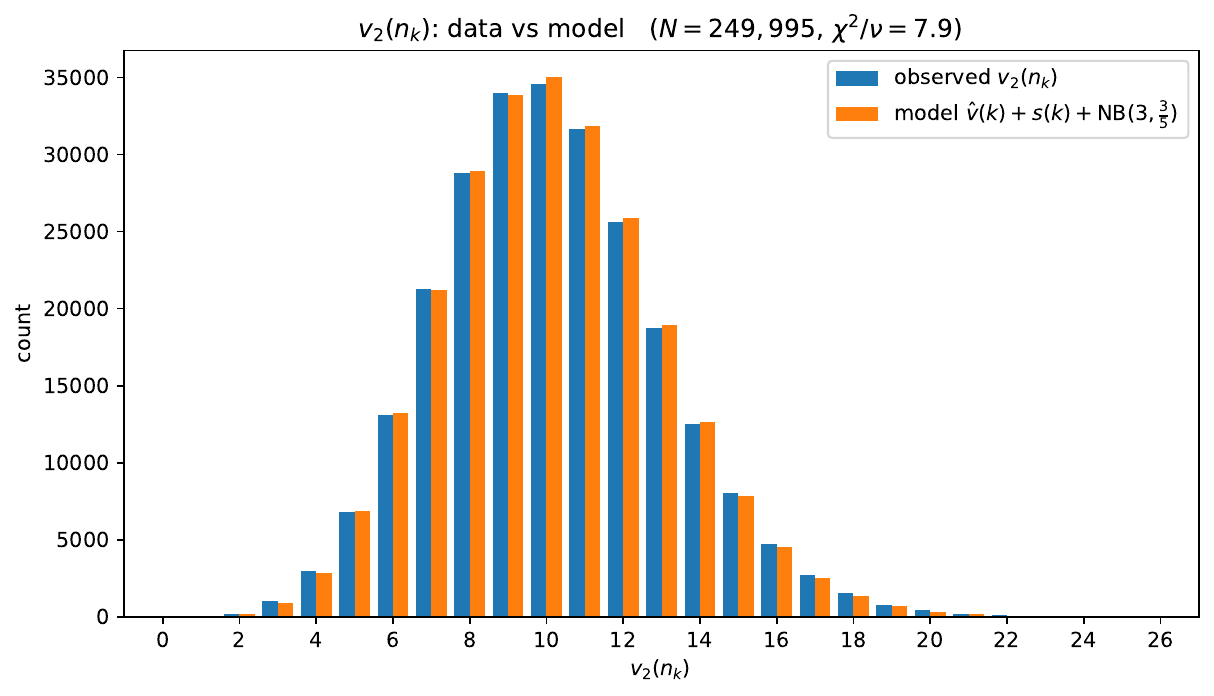}
    \caption{The observed distribution of $v_2(n_k)$
    for values with $v_2(n_k)\geq0$, compared with the
    reference model
    $\hat v(k)+\mathrm{round}(s(k))
    +\mathrm{NB}(3,\tfrac35)$.
    The shift $s(k)$ appearing in the legend is
    understood with this nearest-integer discretization.}
    \label{fig:histogramv2stat}
\end{figure}

For individual valuations, we use the empirical predictor
\begin{equation}\label{eq:P}
\hat P(k):=
\mathrm{round}\bigl(\hat v(k)+s(k)+1\bigr).
\end{equation}
On the sample shown in Figure~\ref{fig:histogramv2stat},
this predictor is the mode of the rounded reference model.
This predictor has Pearson correlation coefficient
$0.78$ with $v_2(n_k)$, only slightly higher than the
correlation obtained using $\hat v(k)$ alone.
Moreover, the prediction $\hat P(k)$ is correct
$24.3\%$ of the time.

\begin{rem}
Since $\hat v(k) + c(k) = \sum_{q\ \mathrm{odd}} W_q(k)$ is  expressed as a sum of local contributions over the odd primes, and the conditions
$q\mid k$ and $q\mid (k^2-16)$ hold on $1$ and $2$ residue classes modulo $q$
(with densities $1/q$ and $2/q$), each odd prime contributes
$1\cdot\tfrac1q + 2\cdot\tfrac2q = \tfrac5q$ to $\hat v$ on average. By Mertens'
theorem, and since $c(k)$ is bounded,
\[
  \frac1x\sum_{5\le k\le x}\hat v(k) = \sum_{q\le x}\frac5q + O(1)
  = 5\log\log (x) + O(1)
\]
as $x \to \infty$.  So, assuming Conjecture \ref{conj:low_bound}, the average order of $v_2(n_k)$ is at least $5 \log \log (x)$.
Tables~\ref{tab:largeW-odd} and \ref{tab:largeW-even}
report $38$ further values in the range
$2.5\times10^5<k<2\times10^6$, selected for large
values of
\[
W(k):=\omega_{\mathrm{odd}}(k)
      +2\,\omega_{\mathrm{odd}}(k^2-16).
\]
These computations provide targeted checks of
Conjecture~\ref{conj:low_bound} beyond the main
dataset, at parameters with many odd primes of
bad reduction.
All $38$ values satisfy $v_2(n_k)\geq\hat v(k)$,
with equality at $k=1540136$, $k=1812794$,
and $k=1984814$.
The value $k=1069814$ gives $v_2(n_k)=27$;
the largest $2$-adic valuation observed in the
main dataset is $31$, attained at $k=36326$.

\end{rem}

\subsection{The regime $k=\sqrt m$}
\label{sec:sqrtm}

For $m < 0$ the parameter $k = \sqrt{m}$ is imaginary and \eqref{HG} does not apply directly; instead we use the functional equation \cite[Theorem~2.2]{LR07} (taking $k=(\sqrt{-m}-\sqrt{16-m})/4$ in their notation).
Using the notation $\mu(k):=\mu_k$, giving
\[\mu(\sqrt m)=\mu\!\left(\frac{64}{(\sqrt{-m}-\sqrt{16-m})^2}\right)-\mu(\sqrt{16-m}),\qquad m<0,\]
allowing us to compute $r_k$ also for $m<0$.

We have computed these $r_k$ for all integers $m$ with
$|m|\leq125{,}000$, excluding the singular parameters
$m=0$ and $m=16$. We briefly compare the observations with integer values of $k$. First of all, for
\[
m \not \in \{ -768, -16, -9, -2, -1, 9, 18, 25, 27, 64, 144, 256 , 1024\},
\]
$r_k$ is the reciprocal of an integer.

For the primes considered in
Table~\ref{tab:chi2-vs-prime}, the observed frequencies
are broadly described by the empirical approximation
$\mathbb P(v_p(n_k)=r)\approx p^{-r}$ for $r\geq1$.
For $p=5$, the proportion with $v_5(n_k)\geq1$ is
$26.24\%$, compared with the reference value $25\%$.
Restricting to the $150{,}000$ parameters for which
$E_k$ has good reduction at $5$, equivalently
$5\nmid m(m-16)$, gives $26.84\%$.
Thus the excess in the proportion of positive
$5$-valuations persists under this restriction.

For $p=3$, the parity relation \eqref{eq:thelawforA}
is satisfied for only about half of the parameters
for which $v_3(n_k)\in\{0,1\}$.
Among the residue classes modulo $27$, the class
$m\equiv25\pmod{27}$ is the only one for which every
computed value with integral $n_k$ satisfies
$v_3(n_k)\geq1$.
This class is the image of $k\equiv\pm5\pmod{27}$.

For $p=2$, a variant of the lower bound in Conjecture \ref{conj:low_bound} holds. The
term $\omega_{\text{odd}}(k)$ appears to split into two different counts on $m$
\[
\omega^0_{\text{odd}}(m) + 2 \omega_{\text{odd}}^1(m),
\]
where $\omega^0_{\text{odd}}(m)$ is the number of odd primes $p$ dividing $m$ with
$v_p(m)$ even, and $\omega^1_{\text{odd}}(m)$ the number of odd primes with $v_p(m)$ odd.
Together with a sign-dependent shift that makes the bound sharp, the generalization of
$\hat{v}(k)$ becomes
\[
\hat{v}(m) = \omega^0_{\text{odd}}(m) + 2 \omega^1_{\text{odd}}(m)
             + 2 \omega_{\text{odd}}(m - 16) - c(m) + \varepsilon(m),
\]
where
\[
c(m) =
  \begin{cases}
    4, & v_2(m) = 0 \text{ or } v_2(m) \geq 8, \\
    2, & 1 \leq v_2(m) \leq 7,
  \end{cases}
\qquad
\varepsilon(m) =
  \begin{cases}
    0, & |m|\ \text{a perfect square},\\
    1, & m<0,\ |m|\, \text{nonsquare},\\
    2, & m>0,\ |m|\, \text{nonsquare}.
  \end{cases}
\]

Then $\hat{v}(m) \leq v_2(n_k)$ appears to hold for
all $m$ in our data for which $n_k$ is integral,
the sole exception being $m=12$.

Equality is attained for $40{,}099$ of the $249{,}986$
parameters for which $n_k$ is integral, or about $16\%$.

Apart from the stated exception, the residual
$v_2(n_k)-\hat{v}(m)$ is non-negative, and its observed
distribution is of the same negative-binomial flavor
as that of the excess $G(k)$ in the integer case.
We further remark that all of the behavior described
above also holds when $m$ is restricted to a single sign.

\subsection{A conditional proof of Conjecture \ref{conj:low_bound} for special families of $k$}
\label{ssec:lowboundproof}
We now use the work of Dummigan, Golyshev, de Jeu and Kerr
\cite{DGJK} to verify Conjecture~\ref{conj:low_bound} conditionally
on the Bloch--Kato conjecture and the additional hypotheses appearing
in their work.  Let
\[
\widetilde E_u\colon y^2=x(x+1)(x+u^2).
\]
The curve $\widetilde E_u$ is $2$-isogenous to $E_{4u}$; hence, when
$k=4u$, the two curves have the same $L$-function and conductor.
\begin{prop}
\label{deJeuprop}
        Let $k = 4 u$, where $u \equiv 4 \pmod 8$ and $u(u^2-1)$ is squarefree at the odd primes; assume in addition that $u+1$ has a prime divisor $\equiv 3 \pmod 4$.
            Assume moreover,
as in \cite[Thm.~3]{DGJK}, that $K_2^T(\tilde E_u)_{\mathrm{int}}\otimes_{\mathbb Z}\mathbb Q$
is one-dimensional, that $r_k$ is a rational number and that $H^1_f(\mathbb Q,\tilde E_u[2^\infty](-1))$ is finite;
and assume the Bloch--Kato conjecture for $\tilde E_u$.

Then Conjecture~\ref{conj:low_bound}, i.e. $\hat v(k)\le v_2(n_k)$, is true for these $k$.
    \end{prop}

\begin{proof}
By the hypotheses on $u$, the hypothesis of \cite[Theorem 1]{DGJK} are satisfied and $m_u = 1$, so that $\iota_{u,2} = 1$. In the notation of \cite{DGJK}, we have $F(u)=\mu_k$ and
\[
q_u=\frac{L(\widetilde E_u,2)}{(2\pi i)^2F(u)}.
\]
By the functional equation and the fact that $\widetilde E_u$ is
isogenous to $E_k$, one has
\[
n_k=-\omega_kN_kq_u.
\]
Under the present hypotheses, the conductor is
\[
N_k=\frac{u(u^2-1)}{4},
\]
which is odd.  Consequently,
\[
v_2(n_k)=v_2(q_u).
\]

    On the one hand, since $\iota_{u,2} = 1$, \cite[Equation (1.3)]{DGJK} reads
    \begin{equation}
    \label{v2_deJeu}
             v_2(n_k) + 2 = 2 \omega_1(u) + \omega_3(u) + \omega(u^2 -1) + s_u,
    \end{equation}
    where
    \[
    s_u = v_2 \left( \#H_f^{1}(\Q, \tilde{E}_{u}[2^\infty](-1))\right),
    \]

     $\omega$ is the number of distinct prime factors, and $\omega_j$ the number of distinct prime factors $\equiv j \pmod 4$. On the other hand, $k = 4u$ implies $\omega_{\text{odd}}(k) = \omega_{\text{odd}}(u) = \omega_1(u) + \omega_3(u)$ and similarly, $\omega_{\text{odd}}(k^2 - 16) = \omega(u^2-1)$. Since  $16 | k$, we have $c(k) = 4$.
Hence
\begin{equation}
\label{hat_in_u}
\hat{v}(k) = \omega_1(u) + \omega_3(u) + 2 \omega(u^2 - 1) - 4 .
\end{equation}
Combining the two Equations \eqref{v2_deJeu} and \eqref{hat_in_u} gives
\[
v_2(n_k) - \hat{v}(k) = s_u - (\omega(u^2 - 1)-\omega_1(u)- 2).
\]
Thus, in order to prove our conjecture for these $k$, it suffices to prove the inequality
\[
s_u \geq \omega(u^2 - 1)-\omega_1(u)- 2.
\]
By elementary group theory we have that
\[
s_u \geq s'_u :=v_2(\#H_f^{1}(\Q, \tilde{E}_{u}[2^\infty](-1))[2]).
\]
We now use the explicit description of this $2$-torsion subgroup given in \cite[Theorem 2]{DGJK}. Let $S$ be the set of prime divisors of $u^2-1$, and let $S'$ be the set of prime divisors of $u$ that are congruent to $1 \pmod 4$. Restricting the parametrization in \cite[Theorem~2]{DGJK} to pairs of the form $(D,1)$ as discussed in \cite[Remark~7.9]{DGJK}, we obtain a subspace consisting of the squarefree divisors $D$ of $u^2-1$ such that
\[
D \text{ is a square modulo every prime in } S'
\qquad\text{and}\qquad
D\equiv 1\pmod 8.
\]
Writing $D=\prod_{p\in S}p^{x_p}$ with $x_p\in \F_2$, the above two conditions give a linear system in the $\omega(u^2 - 1)$ variables $x_p$. Indeed, if $\left( \frac{p}{q} \right) = (-1)^{\alpha_{p, q}}$ for $\alpha_{p, q} \in \mathbb{F}_2$ for each prime $q$ dividing $u$ such that $q \equiv 1 \pmod 4$, we get a system of $\omega_1(u)$ equations $\sum_{p} x_p \alpha_{p,q} = 0$ in $\omega(u^2 - 1)$ variables for the first condition.  The second condition can be expressed as $\left(\frac{-1}{D}\right)=\left(\frac{2}{D}\right)=1$, and this gives at most two extra equations, $\sum_p x_p\beta_{p,-1}=\sum_p x_p\beta_{p,2}=0$, where $\left(\frac{-1}{p}\right)=(-1)^{\beta_{p,-1}}$ and  $\left(\frac{2}{p}\right)=(-1)^{\beta_{p,2}}$.
Thus, the system has $\omega(u^2 - 1)$ variables and at most $\omega_1(u) + 2$ equations. Its solution space therefore has dimension at least   $\omega(u^2-1)-\omega_1(u)-2$. Consequently,
\[
s_u \geq \omega(u^2-1)-\omega_1(u)-2,
\]
and we recover  that statement of Conjecture \ref{conj:low_bound} under these particular conditions.
\end{proof}

\begin{rem}[The Bloch--Kato formula at an arbitrary prime]
\label{rem:general-BK}
The formulation in \cite[\S12]{DGJK} is not restricted to the prime
$2$.
Whenever the hypotheses in Assumption~12.1 of \cite{DGJK} hold, the
Bloch--Kato conjecture predicts
\[
\begin{aligned}
v_p(n_{k})
={}&v_p(N_{k})
+\sum_{q\leq\infty}
v_p\!\left(
\operatorname{Tam}_{q,\omega}^{0}(H^1_{\mathrm{\acute et}}
\bigl(\widetilde E_{u,\overline{\Q}},\Z_p\bigr)(2))
\right)\\
&+
v_p\!\left(
\#H_f^1\bigl(\Q,\widetilde E_u[p^\infty](-1)\bigr)
\right)\\
&-
v_p\!\left(
\#H^0\bigl(\Q,\widetilde E_u[p^\infty](1)\bigr)
\right)
-
v_p\!\left(
\#H^0\bigl(\Q,\widetilde E_u[p^\infty](-1)\bigr)
\right)\\
&-v_p(\iota_{u,p}),
\end{aligned}
\]
where $k=4u$ and $\iota_{u,p}$ is the regulator index defined in
\cite[\S12]{DGJK}. Thus, conjecturally, the valuation of $n_k$
decomposes into conductor and local Tamagawa contributions, a
Tate-twisted Bloch--Kato Selmer contribution, global invariant
terms, and a regulator index.

At present, however, this does not give a practically computable
formula at a general odd prime: neither the full $p$-primary
Bloch--Kato Selmer group nor the regulator index
$\iota_{u,p}$ is known in general.  For example, when $u=4$,
the identity $n_{16}=1/11$ is interpreted in \cite{DGJK} as
predicting that
$
11\mid\iota_{4,11}.
$
\end{rem}

\section{The model} \label{sec:model}

We train a transformer-based model \cite{transformer} using \textsc{Axolver}, a system based on Int2Int \cite{Axolver}, on our dataset. The transformer-based model has 4 encoder layers and 4 decoder layers, 8 attention heads and an embedding dimension of 256 on batches of 32 examples, using the Adam optimizer \cite{adam} with a constant learning rate of $10^{-4}$. All other parameters were left at their default values unless otherwise indicated.

Our input is an integer $k$, encoded as a sequence of digits. The target output is a rational number $r_k$, written as a pair of integers, each again encoded as a sequence of digits. Thus, the task is to learn
\[
k \mapsto r_k
\]
from our dataset.

For the principal experiments, we randomly split our set into
\emph{training}, \emph{validation} and \emph{test sets}, of sizes
$243{,}999$, $3{,}000$ and $3{,}000$, respectively.
We also consider two blocked splits, described in
Section~\ref{sec:secondplateau}.
The training set is used to train the neural network.
The validation set is used to monitor learning over the course
of training and to select the checkpoints reported in the
discussion.
We trained each model from scratch on the training split. At the end
of each epoch, the evaluation pipeline computed and archived
predictions on both the validation and test splits. Checkpoint
selection used only validation accuracy; the test split was not
used for model or checkpoint selection. We report the archived test
scores corresponding to the validation-selected checkpoints.

The learning curves use the validation split. Greedy accuracy is
the percentage of predictions obtained by greedy decoding that
are correct under the task's scoring criterion. Each epoch
consists of $1{,}000$ training steps, and epoch labels start at~$0$.

We attempt to learn $r_k$ in three types of tasks:
learning $r_k$ directly, learning the size of $|r_k|$,
and learning the $p$-adic valuations $v_p(n_k)$.
The tasks involving the learning of $r_k$ and of the
size of $|r_k|$ are performed in base $10$ (as a default),
but the tasks involving the learning of $v_p(n_k)$
are performed in base $p$, since this allows the model
to detect multiplicity by $p$ and the value of $v_p$
more easily than in other bases.

We run the experiments with different features (computed with \textsc{PARI/GP} \cite{PARI}) consisting of various arithmetic invariants of the elliptic curve $E_k$. These are as follows.

\begin{itemize}
\item $\omega_k$, the root number, which is the sign of the functional equation
of $L(E_k,s)$. As explained in the introduction, this value is expected to codify the sign of $r_k$.

\item $N_k$ and $\text{factor}(N_k)$, the conductor of $E_k$ and its factorization. Our goal is to measure whether information on  prime reduction can help the model determine arithmetic properties of $r_k$. We are also interested in seeing if the model can learn the size $|r_k|$ from $N_k$, in view of the estimate \eqref{orderofnk}.

\item $\Delta_k$ and $\text{factor}(\Delta_k)$, the  discriminant of $E_k$ given by the model \eqref{eq:Weierstrass}, and its factorization. Its formula is $\Delta_k=2^{-24}k^{14}(k^2-16)$. Once again, we are interested in studying the connection of prime reduction and the arithmetic properties of $r_k$. Moreover, as equations \eqref{eq:hatvdef} and \eqref{eq:v2approx} indicate, the conjectural estimate for $v_2(n_k)$ seems to involve a count on the primes of bad reduction, with odd primes counted with different weight according to the type of reduction, which is better encoded by the exponents codified in the factorization of the discriminant.

\item $a_p$, the trace of Frobenius for $p=2, 3, 5,
7, 11, 13, 17, 19, 23, 29, 31, 37$, given as a vector. We wish to observe if giving the first traces of Frobenius can help the model learn $r_k$, given that they codify the first coefficients of the $L$-function.

\item The rank bounds for $E_k$, computed using
\texttt{ellrank} in \textsc{PARI/GP} \cite{PARI}.
We use the lower and upper bounds, represented as
the vector
$[\mathrm{rank}_{\min{}},\mathrm{rank}_{\max{}}]$.
These bounds determine the rank when they coincide.
The Birch and Swinnerton-Dyer conjecture relates
the rank to the order of vanishing of $L(E_k,s)$
at $s=1$, which motivates considering this feature
even though our problem concerns the value at $s=2$.

\item $c_2$, the Tamagawa number at $p=2$. We consider this quantity in the learning of $v_2(n_k)$, due to its role in codifying information on the reduction of $E_k$ at $2$, which is very different from the situation of the other primes.

\end{itemize}
The features $\text{factor}(N_k)$ and
$\text{factor}(\Delta_k)$ are encoded as lists of
prime--exponent pairs, with the positive prime factors
in increasing order.
Each list begins with \texttt{N\_Factor} or
\texttt{D\_Factor} and ends with
\texttt{END\_N\_Factor} or \texttt{END\_D\_Factor},
respectively.
Each pair begins with the token \texttt{PAIR},
followed by the prime and its exponent.
Each integer is represented by a sign token
(\texttt{INT+} or \texttt{INT-}), followed by the
digits of its absolute value in the base used for
the task, with one digit per token and the most
significant digit first.
Negative exponents in the factorization of
$\Delta_k$ are retained, and an initial pair
$(-1,1)$ records a negative discriminant.

\section{Machine learning results}\label{sec:machine}

The machine-learning experiments in this section investigate
predictive relationships within the computed range
$1\leq k\leq250{,}000$, with $k\neq4$.
Their purpose is to examine which arithmetic features are useful
for prediction and how the model's predictions relate to the
patterns found in our statistical analysis.
In particular, the improved accuracy at the second $2$-adic
plateau suggests predictive relationships within this range
that are not fully captured by $\hat P(k)$.
This observation does not identify a general arithmetic rule
or establish accuracy for substantially larger values of $k$.

For the valuation tasks at $p=3,5,7$, we compare
with the constant prediction given by the most frequent
value of $v_p(n_k)$ in the training set.
Table~\ref{tab:odd_prime_majority_baselines} reports its
accuracy on the same $3{,}000$ validation examples used
in these experiments.
For $p=5,7$, the constant prediction $0$ has $100\%$
recall on the zero class and $0\%$ recall on every
positive valuation class.
The corresponding comparison for $p=2$, including
balanced accuracy, is reported on the separate test set
in Table~\ref{tab:v2_test_baselines}.

\begin{table}[ht]
    \centering
    \begin{tabular}{cccc}
        \toprule
        $p$ & Constant prediction & Correct predictions & Accuracy \\
        \midrule
        $3$ & $1$ &  $967/3{,}000$ & $32.23\%$ \\
        $5$ & $0$ & $2{,}210/3{,}000$ & $73.67\%$ \\
        $7$ & $0$ & $2{,}546/3{,}000$ & $84.87\%$ \\
        \bottomrule
    \end{tabular}
    \caption{Majority-class baseline accuracies on the
    validation set for the valuation tasks at $p=3,5,7$.
    Each constant prediction is determined from the
    training set.}
    \label{tab:odd_prime_majority_baselines}
\end{table}

\subsection{The size of $|n_k|$} The initial experiments intended for the model to learn $r_k$ given various features discussed in Section \ref{sec:model} (all written in base 10).  However, this approach proved to be too ambitious, as the experiments yielded validation accuracy essentially 0. For these experiments, we noticed that  the model only seems to learn the size of $r_k$ (or equivalently, the size of $n_k$). In the discussion that follows, we will refer to {\em learning the size of $|n_k|$} to the experiment where the model output $s_k$ is considered correct if $s_k$ has the same sign as $r_k$ and if $\frac{|r_k|}{2}< |s_k|< |2r_k|$, that is, the prediction of the {\em factor-two accuracy for $r_k$}.

Including the root number $\omega_k$ as an input feature approximately
doubles the factor-two accuracy of the greedy predictions, see
Figure~\ref{fig:signvsnosign}.
For this comparison, we use three training runs for each choice of input
features. In each run, we select the epoch with the highest validation
factor-two accuracy among the first $300$ epochs.
The corresponding mean test accuracies are $(25.27\pm0.33)\%$ with $k$
alone and $(49.40\pm1.25)\%$ with $(k,\omega_k)$, where the values after
$\pm$ are sample standard deviations over the three runs.
This improvement is consistent with the role of $\omega_k$ in determining
the sign of $r_k$.
We therefore include $\omega_k$ in most of the subsequent experiments.
\begin{figure}
    \centering
    \includegraphics[width=0.65\linewidth]{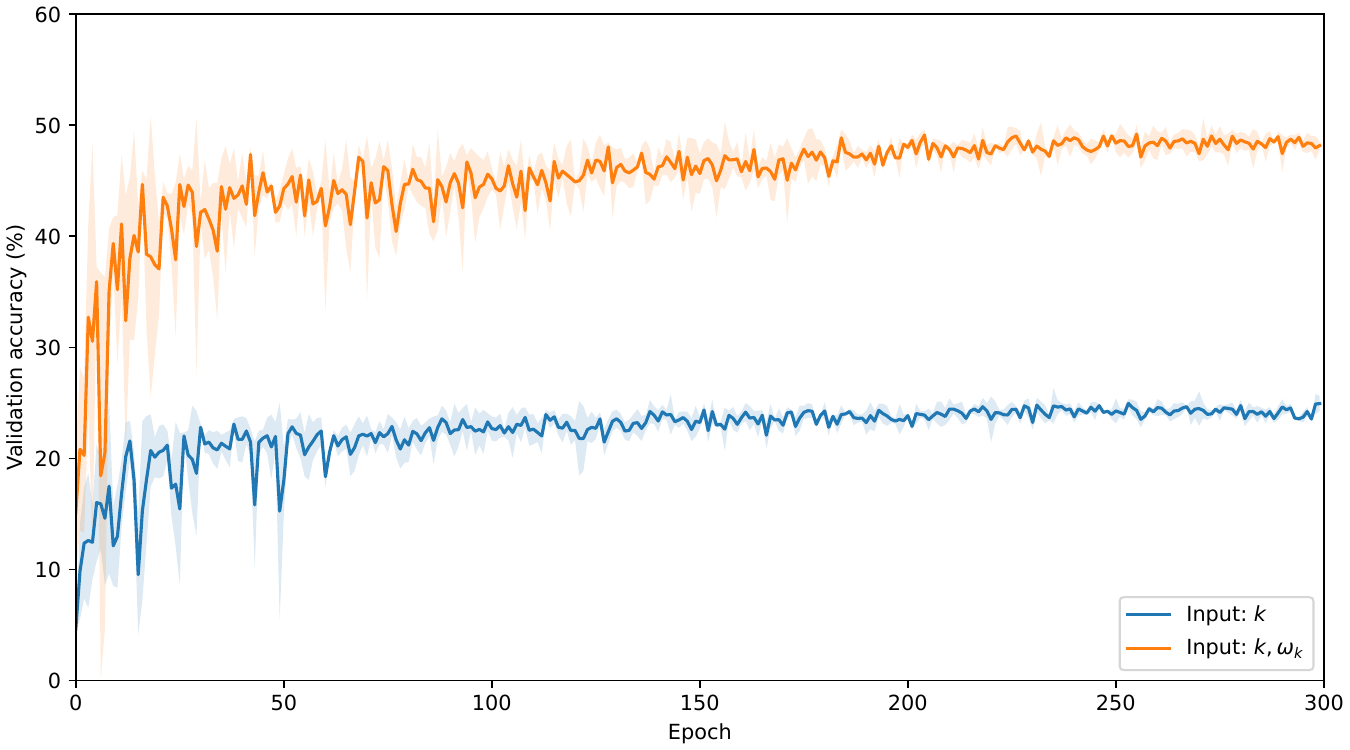}
    \caption{Validation factor-two accuracy for greedy predictions of $r_k$,
    using $k$ alone or $(k,\omega_k)$ as input.
    For each choice of input features, the curve shows the mean over
    three runs and the shaded band shows the mean plus or minus
    one sample standard deviation.
    Log epochs $0,\ldots,299$ are shown.}
    \label{fig:signvsnosign}
\end{figure}

In the context described above, providing the conductor $N_k$ gives
the fastest initial learning, while providing
$\text{factor}(\Delta_k)$ gives faster initial learning than
$\text{factor}(N_k)$, see Figure~\ref{fig:three_extras}.
For each choice of input features, we use three training runs.
In each run, we select the earliest epoch attaining the highest
validation factor-two accuracy among the first $300$ epochs. The corresponding mean test accuracies are $(99.98\pm0.02)\%$
with $N_k$, $(99.42\pm0.14)\%$ with $\text{factor}(N_k)$,
and $(99.22\pm0.17)\%$ with $\text{factor}(\Delta_k)$.
The high accuracy obtained when $N_k$ is supplied is consistent
with \eqref{orderofnk}, which relates the scale of $|n_k|$
to $N_k/\log(k)$.
Moreover, we speculate that the model recovers $N_k$ from $\text{factor}(\Delta_k)$ and $\text{factor}(N_k)$ and
is unable to recover $N_k$ from $\Delta_k$ (due to the fact that factorization is difficult). Nevertheless, we find it puzzling that $\text{factor}(\Delta_k)$ leads to faster learning than $\text{factor}(N_k)$.
The other considered features, namely the rank bounds
and the $a_p$'s, do not seem to influence the model's
learning.

\begin{figure}
    \centering
    \includegraphics[width=0.65\linewidth]{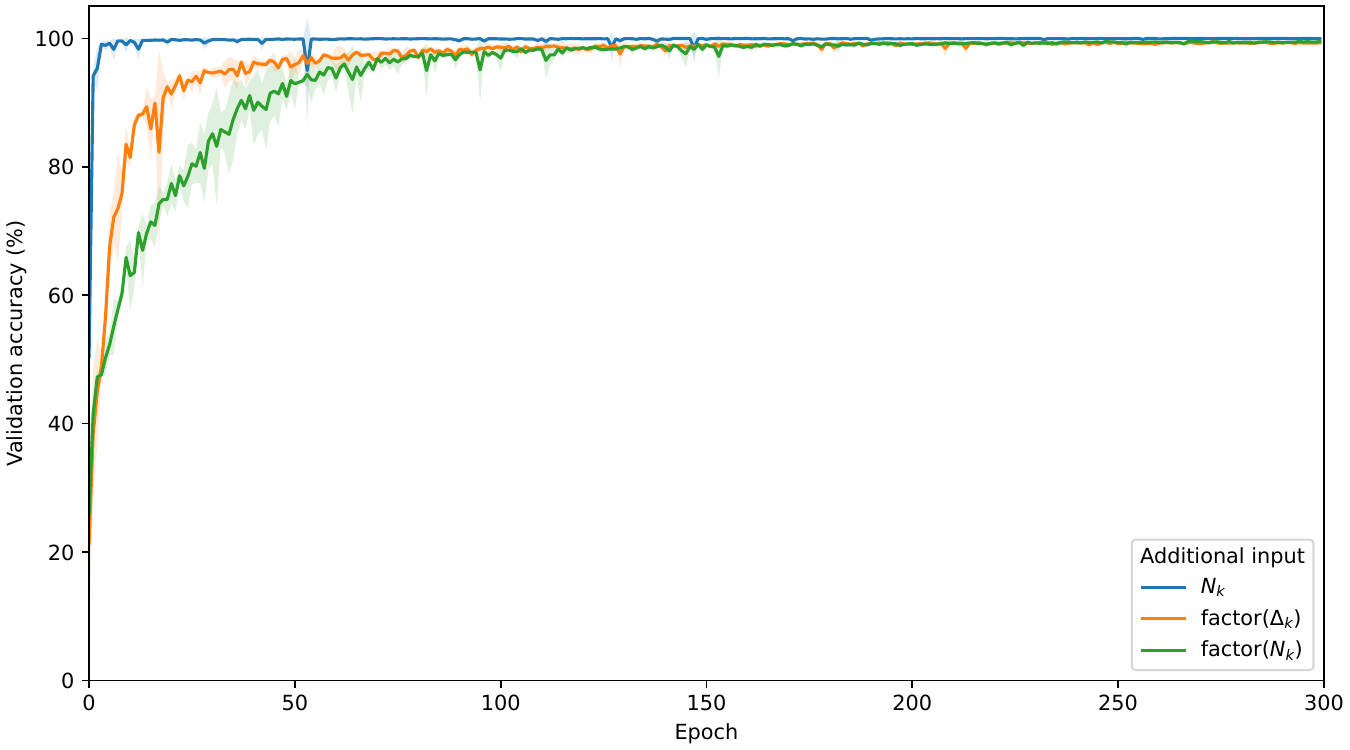}
    \caption{Validation factor-two accuracy for greedy predictions of $r_k$
    from $(k,\omega_k)$ with the additional feature $N_k$,
    $\text{factor}(N_k)$, or $\text{factor}(\Delta_k)$.
    For each choice of input features, the curve shows the mean over
    three runs and the shaded band shows the mean plus or minus
    one sample standard deviation.
    Log epochs $0,\ldots,299$ are shown.}
    \label{fig:three_extras}
\end{figure}

For comparison, replacing $L(E_k,2)$ by $1$ in
\eqref{lvk} leads to the explicit arithmetic predictor
\[
r_k^{\mathrm{base}}
:=\omega_k\frac{4\pi^2\mu_k}{N_k},
\]
where $\mu_k$ is computed from $k$ using \eqref{HG}.
Thus this predictor uses the inputs $k$, $\omega_k$
and $N_k$.
Equations~\eqref{eq:conjecture} and \eqref{lvk} give
\[
\frac{r_k^{\mathrm{base}}}{r_k}=L(E_k,2).
\]
The observed range $(0.5222,1.7532)$ for $L(E_k,2)$,
reported in Section~\ref{sec:statistics}, is contained
in $(1/2,2)$. This implies that the arithmetic
predictor has the correct sign and satisfies the
factor-two criterion throughout the main dataset
$1\leq k\leq250{,}000$, $k\neq4$.
Thus this arithmetic baseline also achieves
$100\%$ factor-two accuracy on these data.
The neural-network experiments show how the
representation of the input features affects
the learning of this scale.

Finally, we use \eqref{eq:conjecture} and \eqref{lvk}
to associate to each nonzero greedy prediction $s_k$ the value
\[
\omega_k\frac{4\pi^2\mu_k}{N_k s_k}.
\]
Figure~\ref{fig:SnapshotofL2} compares the distribution of these
induced values of $L(E_k,2)$ with the reference values reconstructed
from the stored rational candidates $r_k$, on the same $3{,}000$
test examples. We use the three runs with inputs $k$, $\omega_k$
and $\text{factor}(\Delta_k)$, each at its validation-selected epoch.
The central part of the induced distribution resembles the reference
distribution, although a small number of induced values lie outside
the observed reference range.
This comparison concerns marginal distributions and does not by itself
establish accurate pointwise prediction of $L(E_k,2)$.
\begin{figure}
    \centering
    \includegraphics[width=0.65\linewidth]{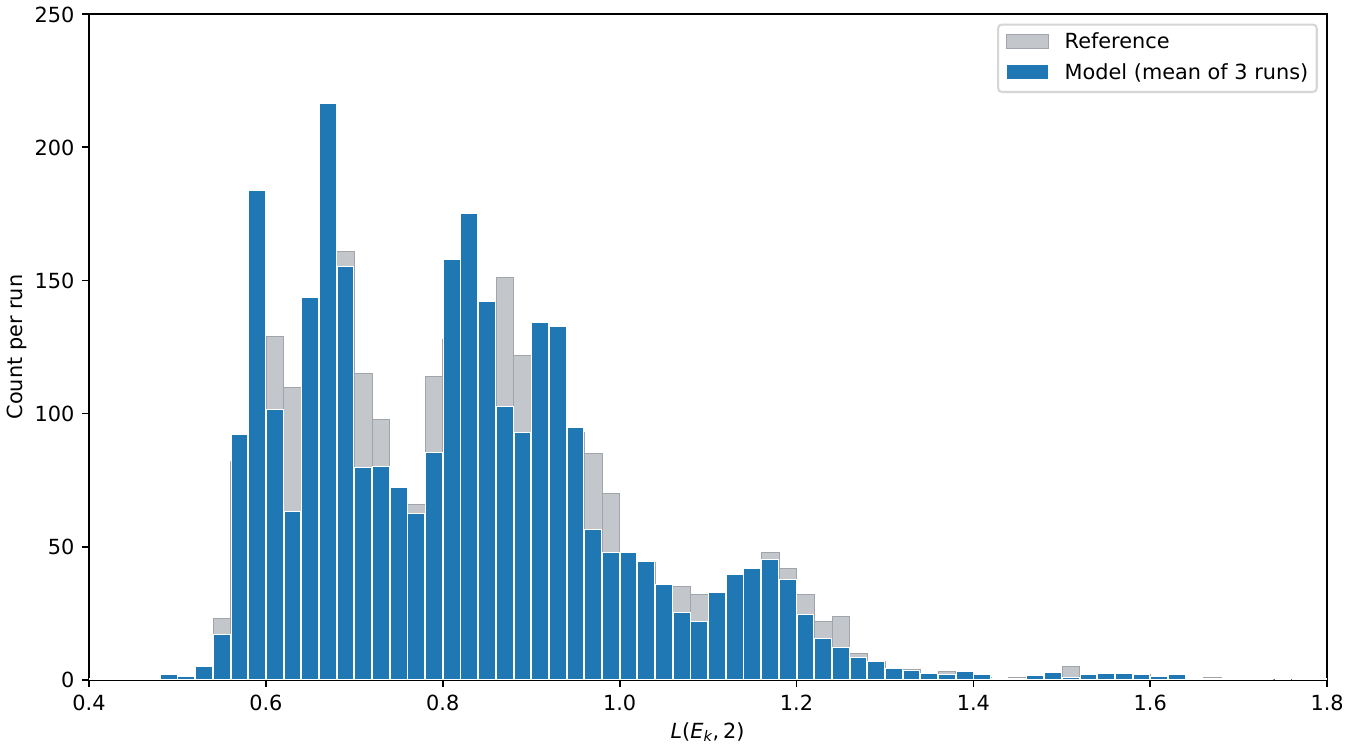}
    \caption{Distributions of the reference and model-induced values of
    $L(E_k,2)$ on the $3{,}000$ test examples, for inputs $k$, $\omega_k$
    and $\text{factor}(\Delta_k)$.
    The grey histogram gives the reference values reconstructed from
    the stored rational candidates $r_k$.
    The blue histogram shows the mean bin counts over three runs,
    at validation-selected log epochs $273$, $125$, and $285$ in run order.
    Both histograms use bins of width $0.02$.
    The displayed interval $[0.4,1.8]$ contains all reference values
    and more than $99\%$ of the induced values in each run.}
    \label{fig:SnapshotofL2}
\end{figure}

\subsection{$p$-adic valuations of $n_k$ for $p \geq 5$.}

These experiments were performed for $p=5,7$.
They had $v_p(n_k)$ as output, with $k$ as the sole
input feature or with the additional feature
$\text{factor}(N_k)$; for $p=5$, we also used
$\text{factor}(\Delta_k)$.
All inputs and outputs were written in base $p$.
The constant prediction $v_p(n_k)=0$ achieves validation accuracies
of $73.67\%$ for $p=5$ and $84.87\%$ for $p=7$. Under the reference
distribution in Section~\ref{subsection:big_p-adic}, the corresponding probability is
$(p-2)/(p-1)$. Two runs for $p=5$ using $\operatorname{factor}(N_k)$
correctly predict some positive valuations, but their overall
validation accuracy at the final epoch is below that of the
constant-zero predictor.

\subsection{$3$-adic valuations of $n_k$}

For $p=3$, we consider three training runs with $k$ and
$\mathrm{factor}(\Delta_k)$ as input features, with all integers
encoded in base $3$.
We express the targets and predictions in terms of
$v_3(n_k)=-v_3(r_k)$.
Figure~\ref{fig:v3accuracy} shows the validation greedy
exact-match accuracy over the course of training.
For each run, we select the earliest epoch attaining the highest
validation accuracy among the first $300$ epochs.
The selected epochs are $244$, $216$, and $298$.
On the separate set of $3{,}000$ test examples, the corresponding
individual exact-match accuracies are $49.20\%$, $48.40\%$, and
$48.80\%$, respectively. The mean accuracy is $(48.80\pm0.40)\%$, where the value
after $\pm$ is the sample standard deviation over the three runs.
On the same test set, the constant prediction $v_3(n_k)=1$
has an accuracy of $31.17\%$.
The validation accuracy declines during the later part of training.

\begin{figure}
    \centering
    \includegraphics[width=0.65\linewidth]{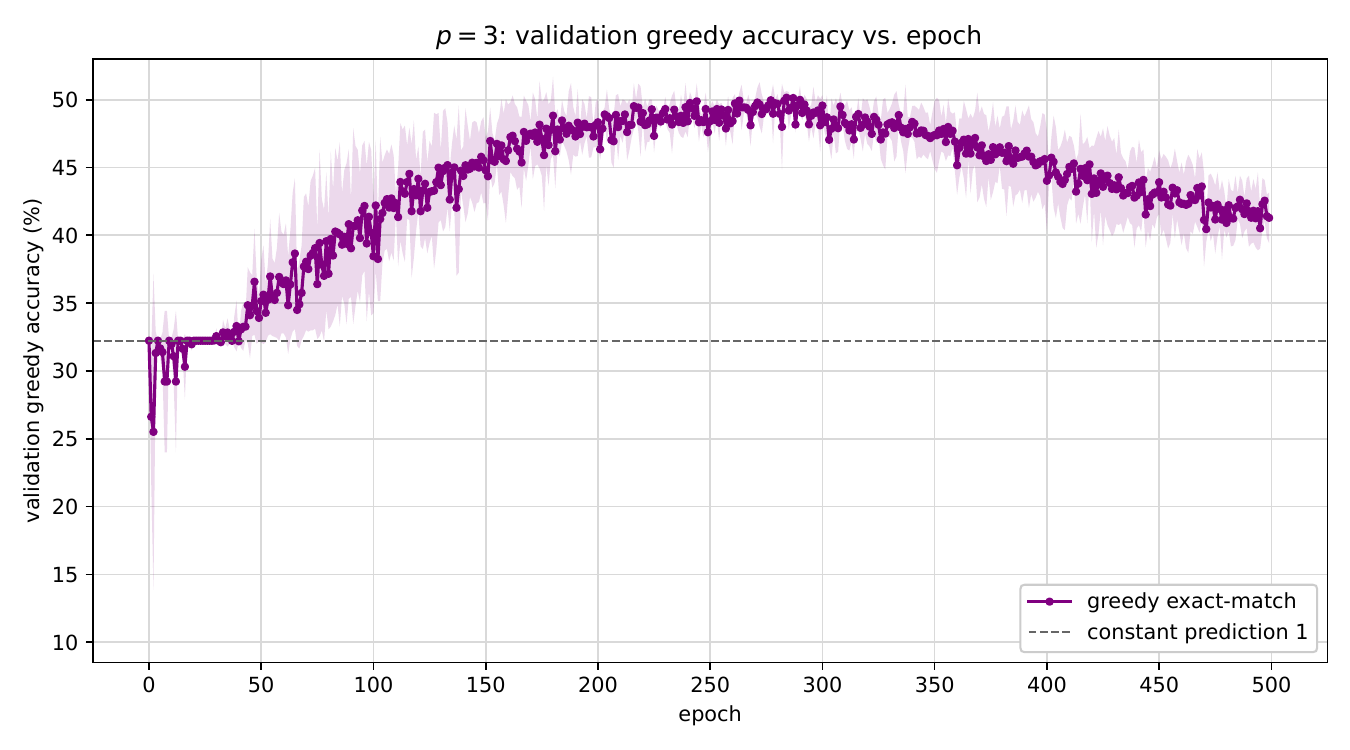}
    \caption{Validation greedy exact-match accuracy for $v_3(n_k)$,
    with inputs $k$ and $\mathrm{factor}(\Delta_k)$ encoded in base $3$.
    The curve shows the mean over three runs, and the shaded band shows
    one sample standard deviation on either side of the mean.
    The dashed line shows the accuracy of the constant prediction $1$
    from Table~\ref{tab:odd_prime_majority_baselines}.
    All $500$ completed training epochs are shown; epoch selection is
    restricted to the first $300$.}
    \label{fig:v3accuracy}
\end{figure}

Figure~\ref{fig:acc_v3_three} shows the validation greedy
exact-match accuracy conditioned on the true valuation
$v_3(n_k)=0$, $v_3(n_k)=1$, or $v_3(n_k)>1$.
At many early epochs, the model predicts $v_3(n_k)=1$ for
every validation example, which explains the initial plateau in
Figure~\ref{fig:v3accuracy}.
The improvement over this constant prediction comes mainly
from learning to recognize the cases with $v_3(n_k)=0$.
At the validation-selected epochs used above, the mean
conditional exact-match accuracies on the test set are $(68.71\pm1.23)\%$ for $v_3(n_k)=0$,
$(95.40\pm0.81)\%$ for $v_3(n_k)=1$, and
$(1.77\pm0.92)\%$ for $v_3(n_k)>1$.
Thus, exact prediction of the higher valuations remains poor.
During the later part of training, the validation accuracy
on the $v_3(n_k)=1$ cases declines, contributing to the
decrease in overall accuracy seen in Figure~\ref{fig:v3accuracy}.

\begin{figure}
    \centering
    \includegraphics[width=0.65\linewidth]{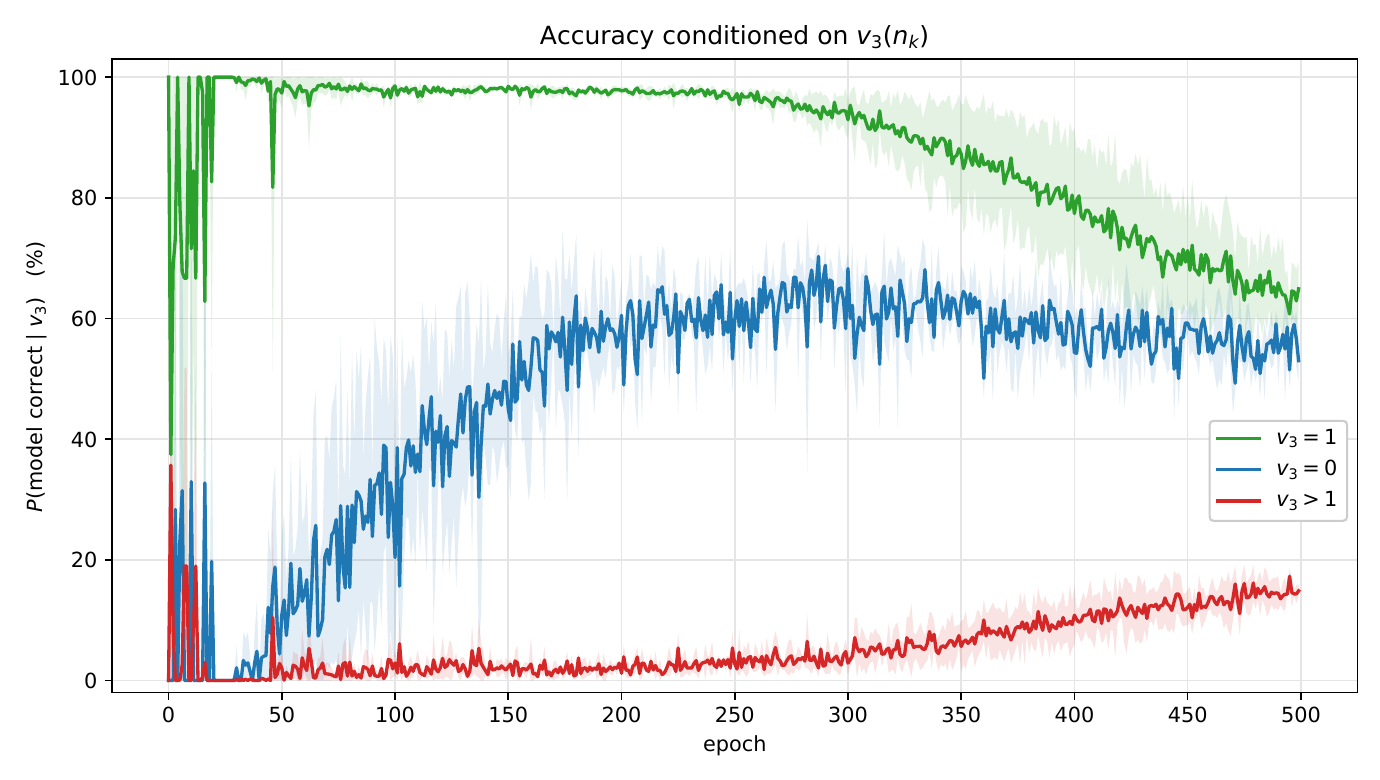}
    \caption{Validation greedy exact-match accuracy conditioned on
    the true valuation $v_3(n_k)$.
    The green, blue and red curves correspond to $v_3(n_k)=1$,
    $v_3(n_k)=0$ and $v_3(n_k)>1$, respectively.
    Each curve shows the mean over three runs; the shaded bands show
    one sample standard deviation, restricted to $[0,100]\%$ for display.
    In the group $v_3(n_k)>1$, a prediction is counted as correct
    only if it gives the exact valuation.
    All $500$ completed training epochs are shown.}
    \label{fig:acc_v3_three}
\end{figure}

\subsubsection{Learning on different congruence classes}

In our statistical analysis, we observed that $v_3(n_k)\geq1$
for $k\equiv0,\pm5\pmod{27}$, apart from a few exceptions
in the dataset. At the validation-selected epochs used above, the model
predicts $v_3(n_k)=1$ with a mean frequency of
$(97.27\pm2.44)\%$ on the $354$ test examples in these classes.
Its mean exact-match accuracy on these examples is
$(41.90\pm0.16)\%$, compared with $42.09\%$ for the
constant prediction $1$ and $(48.80\pm0.40)\%$ on the
full test set.
Thus, at these selected epochs, accuracy on these congruence
classes is very close to that of the constant predictor.

\subsubsection{Learning on the sets from Table \ref{tab:sets}}

Next we examine the model's predictions on the set $A$,
for which the empirical parity laws~\eqref{eq:law-general}
were discussed above.
Figure~\ref{fig:v3_three_A} shows the validation greedy
exact-match accuracy for $k\in A$, conditioned on the true
valuation $v_3(n_k)=0$, $v_3(n_k)=1$, or $v_3(n_k)>1$.

\begin{figure}
    \centering
    \includegraphics[width=0.65\linewidth]{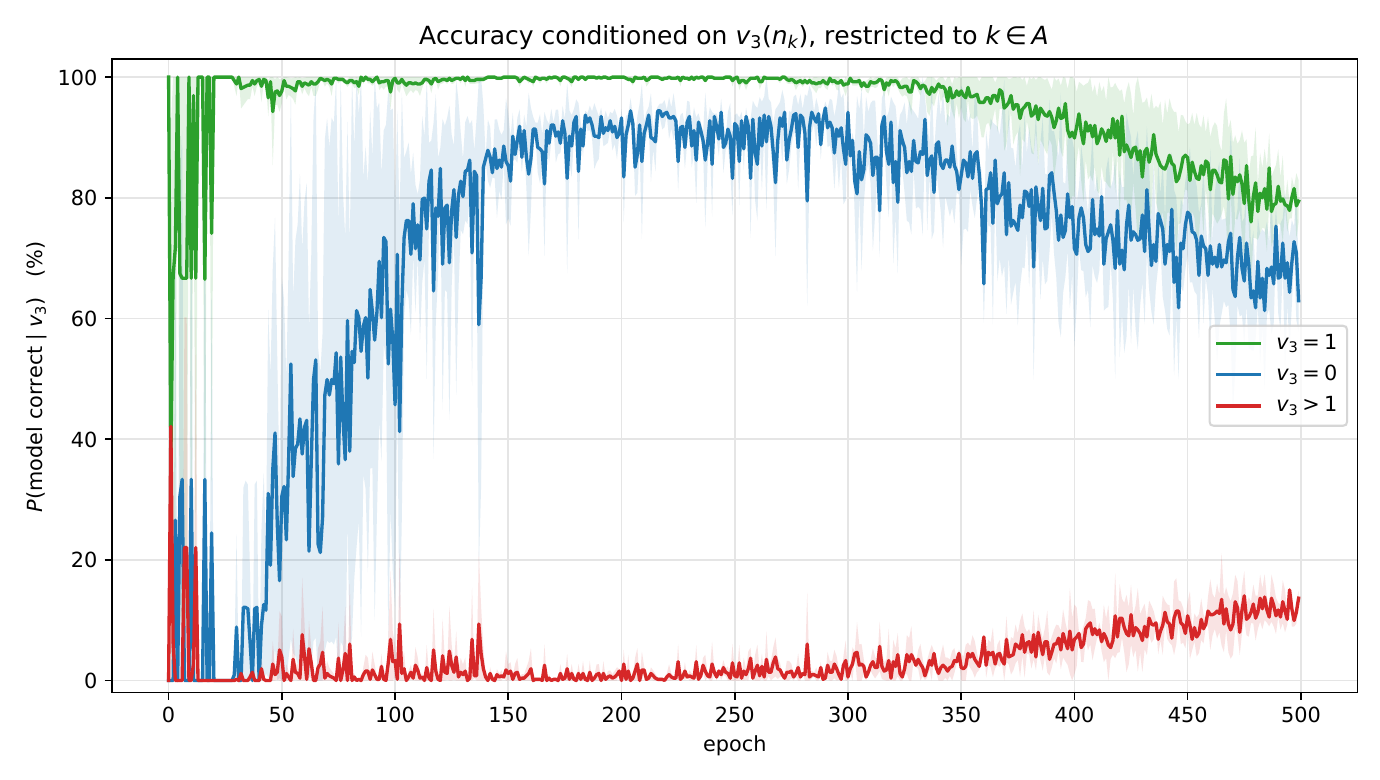}
    \caption{Validation greedy exact-match accuracy on $k\in A$,
    conditioned on the true valuation $v_3(n_k)$.
    The green, blue and red curves correspond to $v_3(n_k)=1$,
    $v_3(n_k)=0$ and $v_3(n_k)>1$, respectively.
    Each curve shows the mean over three runs; the shaded bands
    show one sample standard deviation, restricted to $[0,100]\%$
    for display.
    For $v_3(n_k)>1$, correctness still requires the exact valuation.
    All $500$ completed training epochs are shown.}
    \label{fig:v3_three_A}
\end{figure}

At the validation-selected epochs used above, the mean conditional test accuracies on $A$ are $(96.60\pm0.37)\%$
for $v_3(n_k)=0$ and $(99.78\pm0.39)\%$ for $v_3(n_k)=1$.
On $A$, the model therefore distinguishes the cases with
valuation $0$ from those with valuation $1$ with high accuracy.
For the $180$ test examples in $A$ with $v_3(n_k)>1$, the
mean exact-match accuracy at these selected epochs is
$(0.19\pm0.32)\%$.
Figure~\ref{fig:acc-by-set} compares greedy exact-match
accuracy on the sets in Table~\ref{tab:sets}, restricted
to examples with true $v_3(n_k)\in\{0,1\}$.
The table reports test results at the validation-selected epochs
used above, and the curves show mean validation accuracy
over the three runs.

\begin{figure}[t]
\centering
\begin{minipage}[c]{0.30\textwidth}
  \centering\small
  \renewcommand{\arraystretch}{1.15}
  \begin{tabular}{@{}l r r@{}}
  \toprule
  set & $n$ & acc. (\%)\\
  \midrule
  $A$   & $306$ & $98.15\pm0.38$\\
  $A'$  & $340$ & $95.88\pm1.02$\\
  $A''$ & $566$ & $92.82\pm1.30$\\
  \midrule
  $B$   & $1{,}133$ & $86.82\pm0.59$\\
  $B'$  & $977$ & $91.98\pm0.97$\\
  $B''$ & $937$ & $91.89\pm1.08$\\
  \midrule
  $C$   & $1{,}251$ & $85.74\pm0.60$\\
  $C'$  & $1{,}081$ & $90.63\pm0.94$\\
  $C''$ & $1{,}035$ & $90.43\pm1.05$\\
  \bottomrule
  \end{tabular}
\end{minipage}\hfill
\begin{minipage}[c]{0.66\textwidth}
  \centering
  \includegraphics[width=\linewidth]{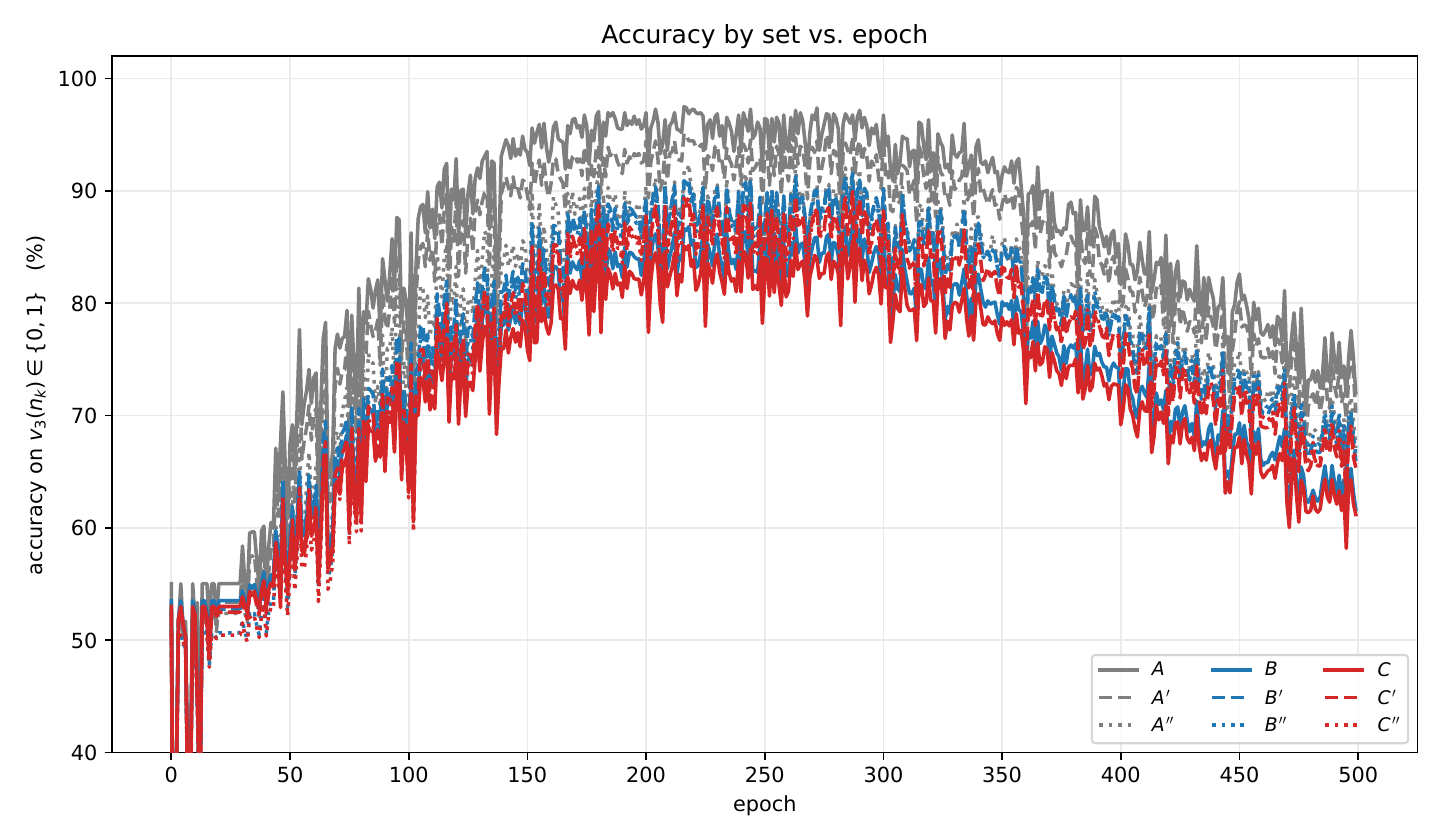}
\end{minipage}
\caption{Greedy exact-match accuracy restricted to true
$v_3(n_k)\in\{0,1\}$ on the sets in Table~\ref{tab:sets}.
Left: mean test accuracy, with sample standard deviation
over three runs, at the validation-selected epochs used above;
$n$ is the number of test examples in each set with true
$v_3(n_k)\in\{0,1\}$.
Right: mean validation accuracy over the same three runs, as a function of the
 training-log epoch, numbered from $0$. The vertical axis starts at $40\%$.}
\label{fig:acc-by-set}
\end{figure}

In this comparison, the highest mean test accuracy is attained
on the set $A$.
Although Table~\ref{tab:sets} records no exceptions to the parity
laws on some of the larger sets, these sets admit more varied
factorizations of $k(k^2-16)$.
We speculate that this makes the relevant parity information
more difficult for the model to extract.
Excluding the values $k\equiv4\pmod8$, passing from $B$ to $B'$
and from $C$ to $C'$, increases the mean conditional test accuracy.

Finally, we consider the values of $k$ outside
$C'' \cup \{k \equiv 0,\pm5 \pmod{27}\}$, restricting to those
for which $v_3(n_k)\in\{0,1\}$.
Figure~\ref{fig:outside} shows the mean validation accuracy on these
examples. After an initial increase, the accuracy declines in the later
stages of training. At the validation-selected epochs used above, the
mean test accuracy on the $551$ such examples is
$(65.34\pm0.96)\%$. On the same test examples, the constant
predictor $v_3(n_k)=1$ has accuracy $49.91\%$.
Thus the model also achieves accuracy above this baseline on the
remaining examples within the computed range.

\begin{figure}
    \centering
    \includegraphics[width=0.65\linewidth]{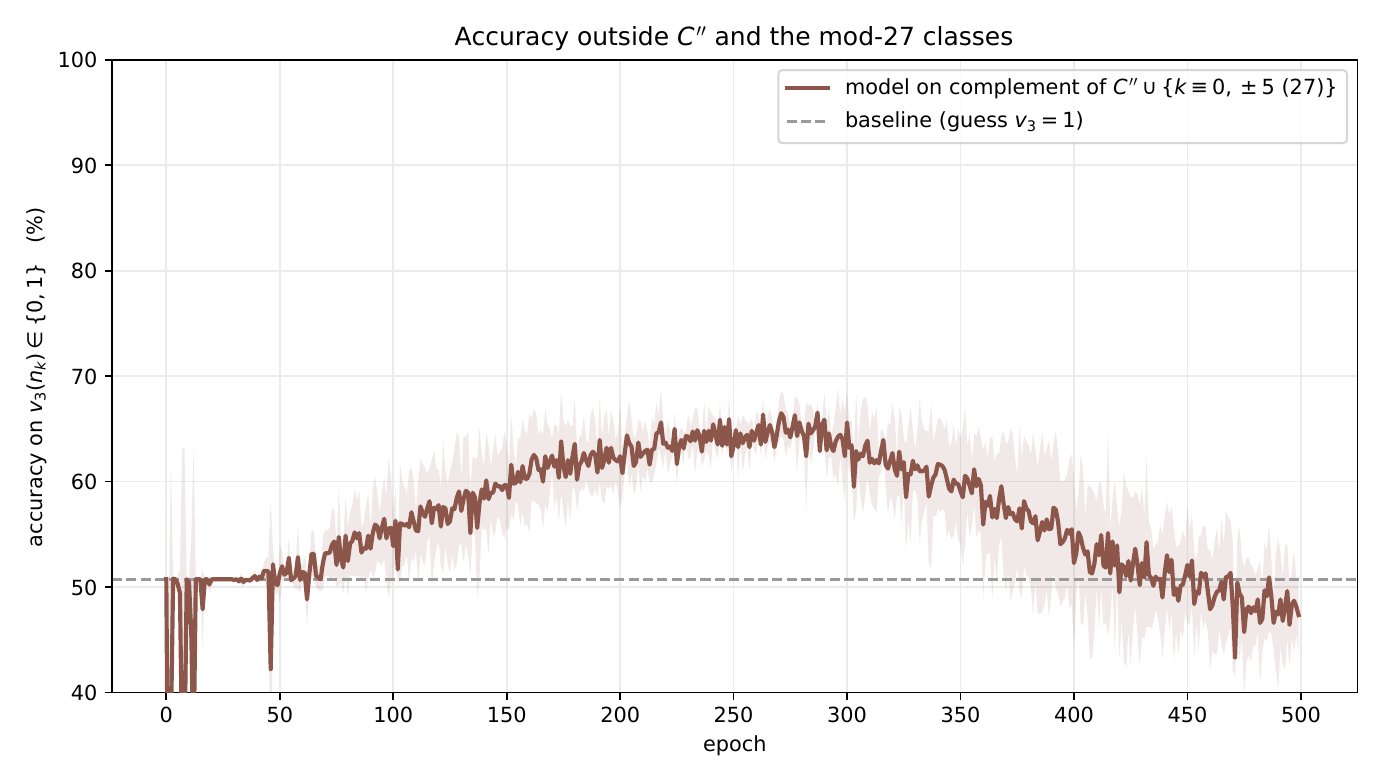}
    \caption{Validation accuracy outside
    $C'' \cup \{k \equiv 0,\pm5 \pmod{27}\}$, restricted to the
    $540$ validation examples with true valuation
    $v_3(n_k)\in\{0,1\}$. The brown curve is the mean greedy exact
    accuracy over three runs, and the shaded band shows one sample
    standard deviation on either side of the mean. The dashed grey
    line is the accuracy of the constant predictor $v_3(n_k)=1$
    on the same validation examples, namely $50.74\%$.
    The vertical axis starts at $40\%$.}
    \label{fig:outside}
\end{figure}

Using the same rule for selecting epochs from validation accuracy
in each configuration, the three runs with inputs $k$ and
$\text{factor}(N_k)$ give a mean greedy exact-match accuracy of $(47.06\pm0.30)\%$ on the full test set of $3{,}000$
examples, compared with $(48.80\pm0.40)\%$ for inputs
$k$ and $\text{factor}(\Delta_k)$.
With the conductor factorization, mean validation accuracy
rises earlier in training: at epoch $50$ it is $43.74\%$,
compared with $35.16\%$ for the discriminant factorization.
Using $k$, $\text{factor}(N_k)$, and $\text{factor}(\Delta_k)$
together as inputs gives a mean greedy exact-match test
accuracy of $(47.49\pm0.59)\%$ under the same validation-based
epoch-selection rule.
At epoch $50$, mean validation accuracy for this configuration
is $37.74\%$, compared with $43.74\%$ for inputs $k$ and
$\text{factor}(N_k)$.
By contrast, supplying $N_k$ or $\Delta_k$ themselves, rather than
their factorizations, does not lead to appreciable learning in these
experiments. In both configurations, the validation accuracy remains
essentially at the majority-class baseline obtained by always predicting
$v_3(n_k)=1$. At the validation-selected checkpoints, the mean test
accuracies are $(31.17\pm0.00)\%$ with $N_k$ and
$(31.21\pm0.08)\%$ with $\Delta_k$, compared with $31.17\%$ for
the constant prediction $v_3(n_k)=1$ on the same test set.
This is consistent with the factorization of these quantities, rather
than merely their size, being the useful arithmetic input for this task.

\subsection{$2$-adic valuations of $n_k$}

For $p=2$, we consider three runs in which the model is given $k$ and
the factorization of the discriminant $\mathrm{factor}(\Delta_k)$ as
input and $v_2(n_k)$ as output, all written in base $2$.
Figure~\ref{fig:v2_val_acc} shows the mean validation accuracy over
these runs. The model quickly reaches an accuracy of about
$23\%-25\%$, giving a first plateau during the early stages of
training. The accuracy then begins to rise around epochs $80$--$100$
and reaches a second plateau around $34\%-35\%$.
Using the same validation-based epoch-selection rule as above, the
selected epochs for the three runs are $261$, $272$, and $211$.
The corresponding mean accuracy on the full test set of $3{,}000$
examples is $(37.12\pm1.60)\%$, where the value after $\pm$ is the
sample standard deviation over the three runs.
We analyze the behavior of the model at the two plateaus below.

\begin{figure}
    \centering
    \includegraphics[width=0.65\linewidth]{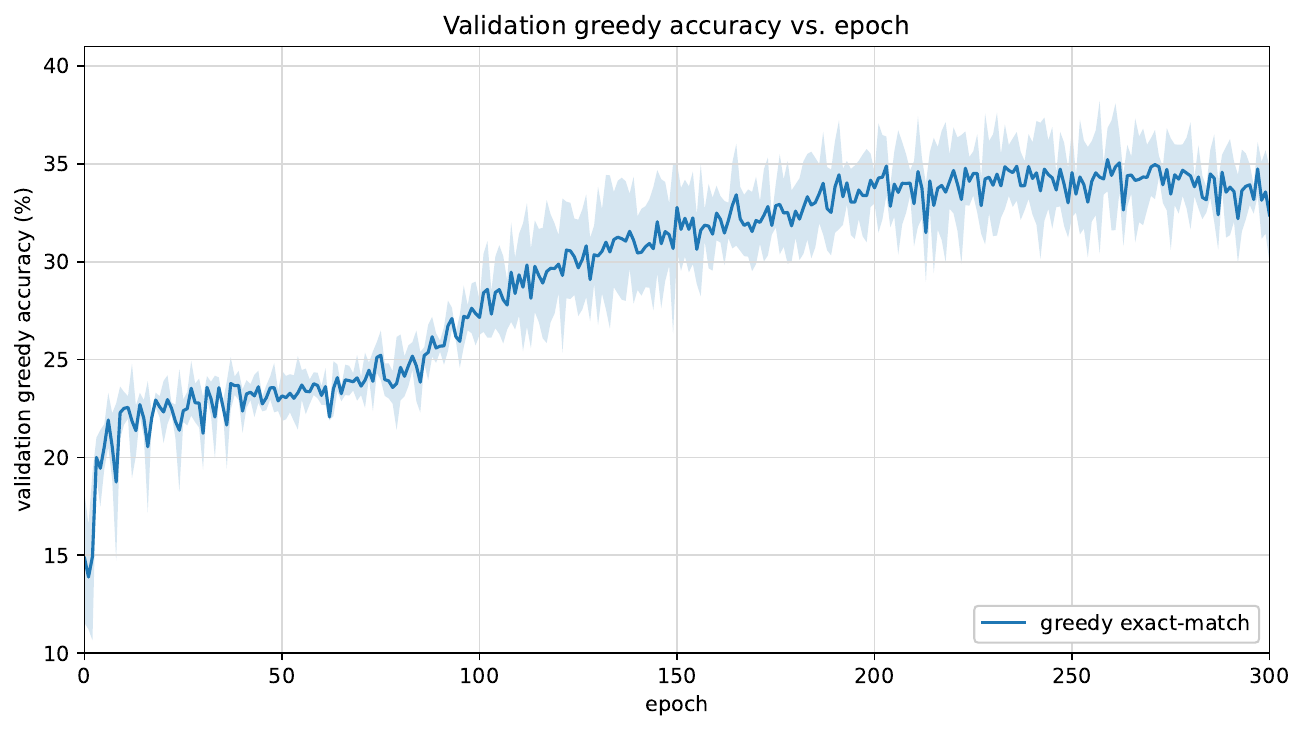}
    \caption{Validation greedy exact-match accuracy for the $2$-adic
    experiment with inputs $k$ and $\mathrm{factor}(\Delta_k)$.
    The curve is the mean over three runs, and the shaded region shows
    one sample standard deviation on either side of the mean.}
    \label{fig:v2_val_acc}
\end{figure}

For each of the three runs, we select the epoch with the highest
validation greedy exact-match accuracy among the first $300$ epochs.
The selected epochs are $261$, $272$, and $211$.
We report the archived predictions for these checkpoints on the separate test set of
$3{,}000$ examples, which is not used for epoch selection.
The corresponding test accuracies are $37.63\%$, $35.33\%$, and $38.40\%$, giving a mean accuracy of
$(37.12\pm1.60)\%$.

On the same test set, the predictor $\hat P(k)$ is correct on $737$ examples, giving an accuracy of $24.57\%$.
As a majority-class baseline, we use the constant prediction $10$,
the most frequent value of $v_2(n_k)$ in the training set.
It is correct on $392$ of the $3{,}000$ test examples, giving an
accuracy of $13.07\%$.

We also report balanced accuracy: for each true valuation present
in the test set, we compute the proportion of correct predictions
and then take the unweighted average over these classes.
The $23$ classes are $v_2(n_k)=2,\ldots,24$.
The comparison is given in Table~\ref{tab:v2_test_baselines}.

\begin{table}[ht]
    \centering
    \begin{tabular}{lcc}
        \toprule
        Predictor & Exact-match accuracy & Balanced accuracy \\
        \midrule
        Model (three runs) &
        $(37.12\pm1.60)\%$ & $(27.40\pm0.32)\%$ \\
        $\hat{P}(k)$ &
        $24.57\%$ & $13.08\%$ \\
        Constant prediction $10$ &
        $13.07\%$ & $4.35\%$ \\
        \bottomrule
    \end{tabular}
    \caption{Exact-match and balanced accuracies on the separate
    test set of $3{,}000$ examples. For the model, we report the
    mean and sample standard deviation over three runs, with the
    checkpoint in each run selected using validation accuracy.}
    \label{tab:v2_test_baselines}
\end{table}

\subsubsection{Learning around the first plateau of Figure~\ref{fig:v2_val_acc}.}
We first investigate whether the model learns our conjectural lower
bound $\hat{v}(k)\leq v_2(n_k)$ from Conjecture~\ref{conj:low_bound}. Figure~\ref{fig:v2lowerbound} shows the mean proportion of validation
examples for which the model's predicted valuation is at least
$\hat v(k)$.
The proportion rises very quickly: by epoch~$7$ its mean value is $99.79\%$, and from that point through epoch~$21$ it remains above
$96\%$ and is usually close to $100\%$.
At the validation-selected checkpoints used above, the corresponding
mean proportion on the separate test set is $(99.86\pm0.08)\%$.
\begin{figure}
    \centering
    \includegraphics[width=0.65\linewidth]{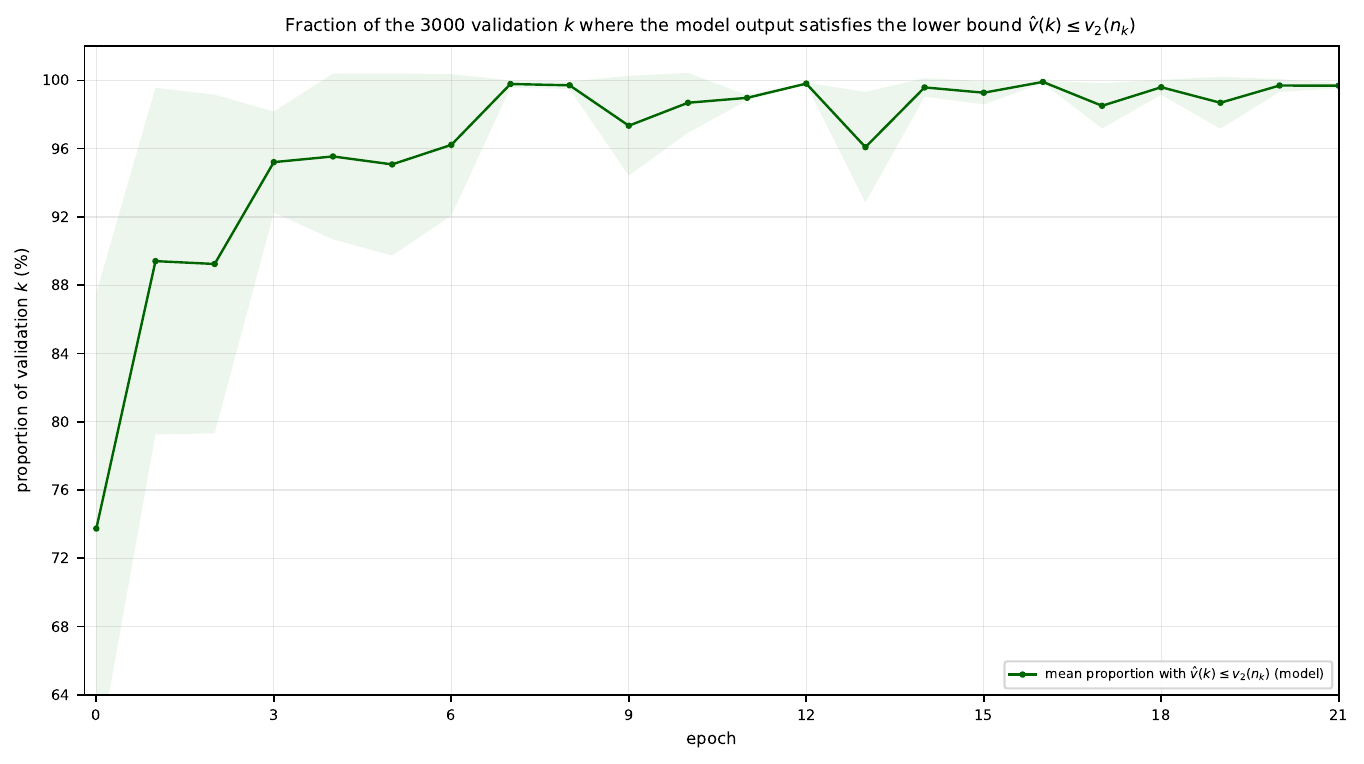}
    \caption{The proportion of $k$ in the validation set for which
the model's predicted valuation is at least $\hat v(k)$.
The green curve is the mean over three runs, and the shaded region
shows one sample standard deviation on either side of the mean.}
\label{fig:v2lowerbound}
\end{figure}

 In particular, we compare our predictor
$\hat{P}(k)=\mathrm{round}(\hat{v}(k)+s(k)+1)$,
given in \eqref{eq:P}, with the model's predictions on the
validation set by computing their Pearson correlation coefficient
at each epoch; see Figure~\ref{fig:pearson_pred}.
The mean correlation over the three runs rises rapidly and reaches
its maximum, $r\approx0.971$, at epoch~$62$.
It then decreases overall, reaching $r\approx0.876$ at epoch~$300$.
At the validation-selected checkpoints used above, the mean
correlation is $0.918\pm0.008$.

The increase in validation accuracy after the first plateau,
together with the decrease in correlation with $\hat{P}(k)$,
suggests that the model improves its predictions while gradually
departing from this explicit predictor.
These observations do not identify the arithmetic features
responsible for the improvement.
For contrast, the correlation between the true valuations
$v_2(n_k)$ and our predictor over the full dataset of
$249{,}999$ values is only $r=0.78$.
Figure~\ref{fig:confusion_snapshots} gives another view of the
same phenomenon.

For each displayed epoch, we pool the validation predictions from
the three runs. The selected-checkpoint panel pools the validation
predictions from each run's own validation-selected checkpoint.
The upper subfigure compares the model's predictions with $\hat{P}(k)$,
while the lower subfigure compares them with the true valuations $v_2(n_k)$.

The epoch~$40$ panel for $\hat{P}(k)$ shows a strong concentration
near the diagonal during the first plateau. The pooled Pearson
correlation with $\hat{P}(k)$ is $0.94$ at epoch~$100$ and $0.87$
at epoch~$316$. The corresponding correlations with the true
valuations are $0.74$ and $0.71$, respectively.
Thus the increase in accuracy beyond the first plateau is not
explained simply by the model becoming a closer approximation
to $\hat{P}(k)$.

\begin{figure}[p]
    \centering
    \begin{subfigure}{0.89\textwidth}
        \centering
        \includegraphics[width=\textwidth]{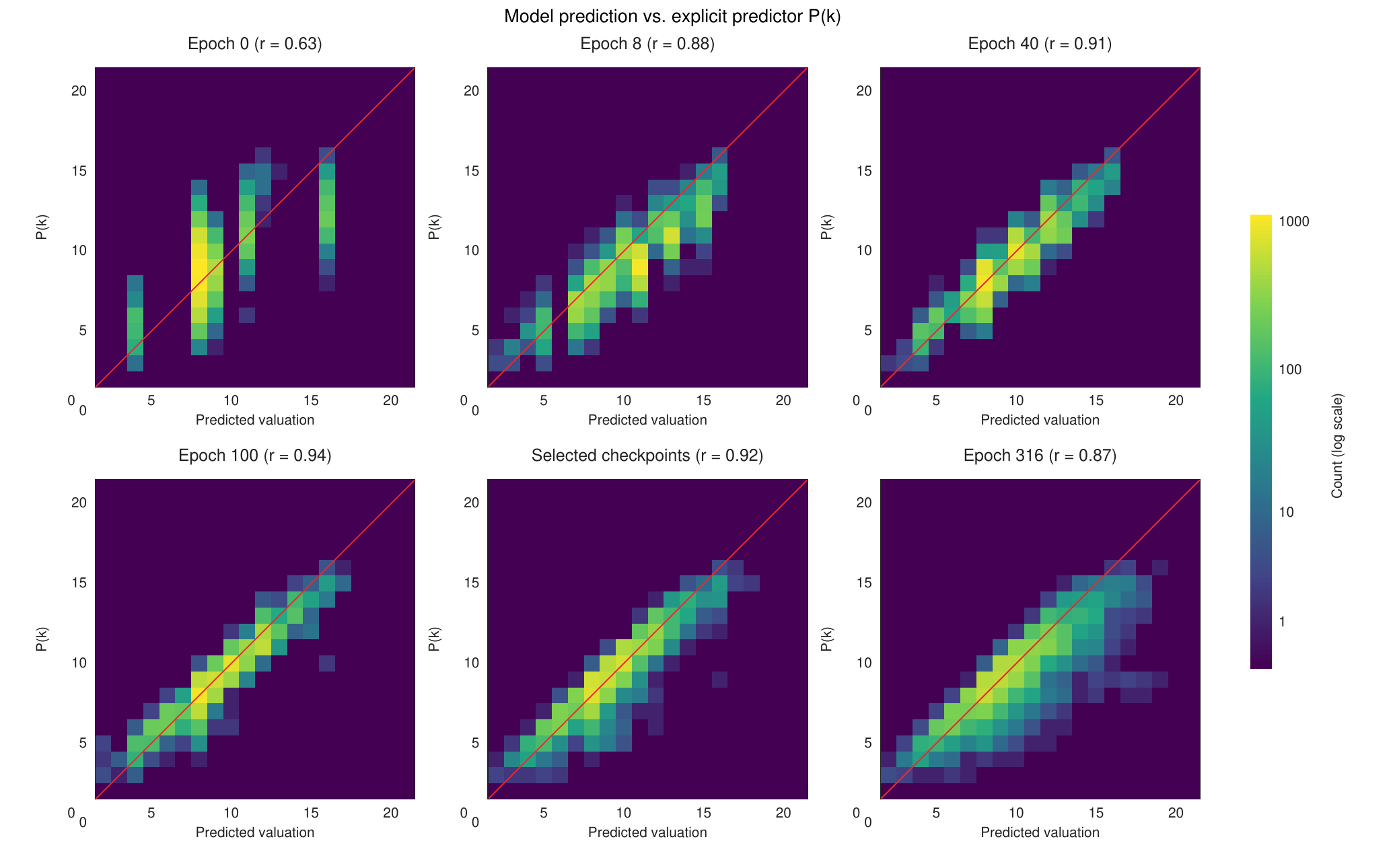}
        \caption{Model prediction vs. predictor
        $\hat{P}(k)=\mathrm{round}(\hat{v}(k)+s(k)+1)$.}
        \label{fig:predictor_vs_model}
    \end{subfigure}

    \vspace{1em}

    \begin{subfigure}{0.89\textwidth}
        \centering
        \includegraphics[width=\textwidth]{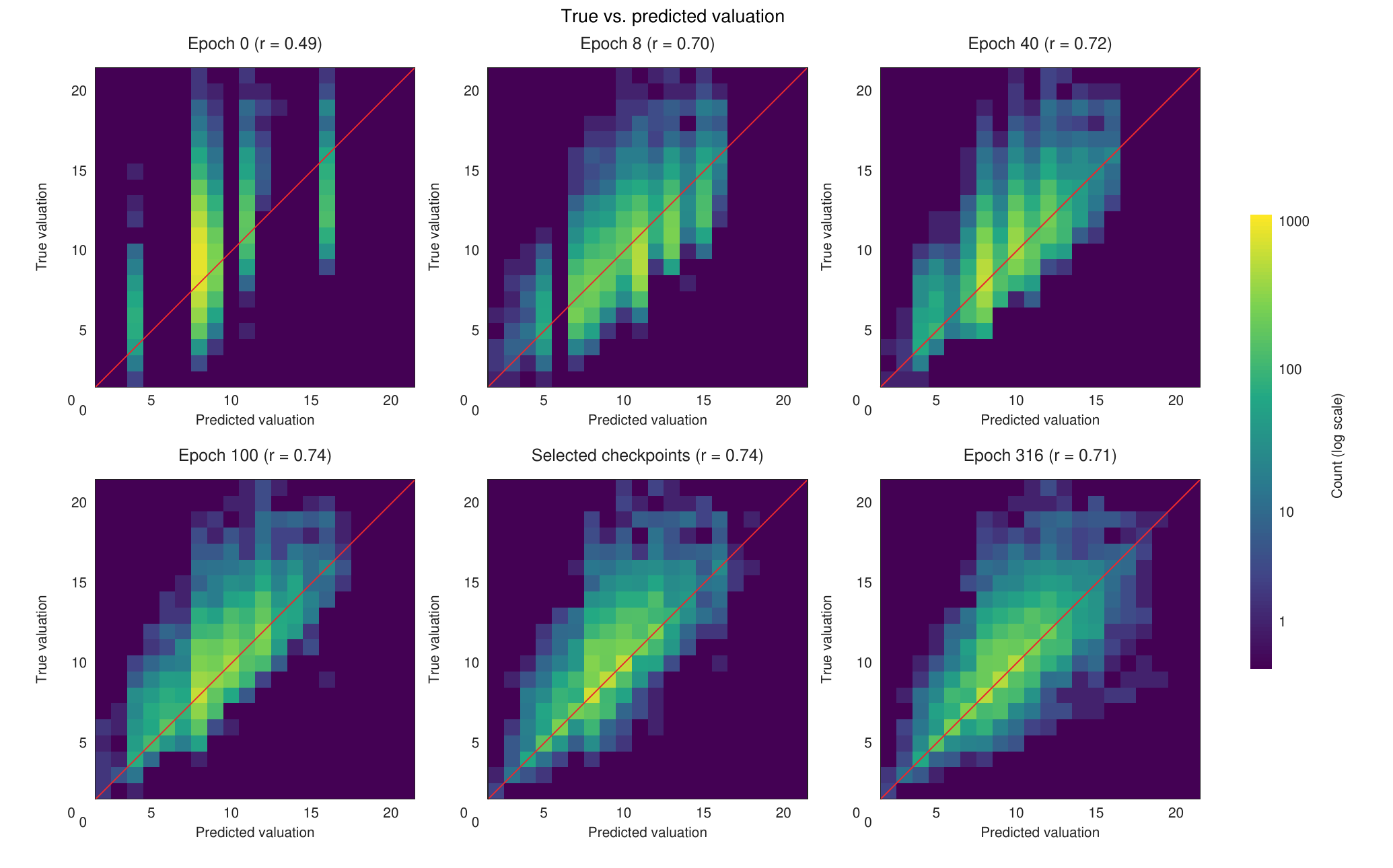}
        \caption{Model prediction vs. true $v_2(n_k)$.}
        \label{fig:truth_vs_model}
    \end{subfigure}

    \caption{Confusion matrices at epochs $0$, $8$, $40$, $100$,
    and $316$, and at the validation-selected checkpoints
    (epochs $261$, $272$, and $211$ in run order).
    Each panel pools up to $9{,}000$ valid predictions from three
    runs on the same $3{,}000$ validation examples. The color scales
    show counts on a logarithmic scale. Panel titles give the
    pooled Pearson correlation coefficient~$r$.}
    \label{fig:confusion_snapshots}
\end{figure}

\begin{figure}
    \centering
    \includegraphics[width=0.65\linewidth]{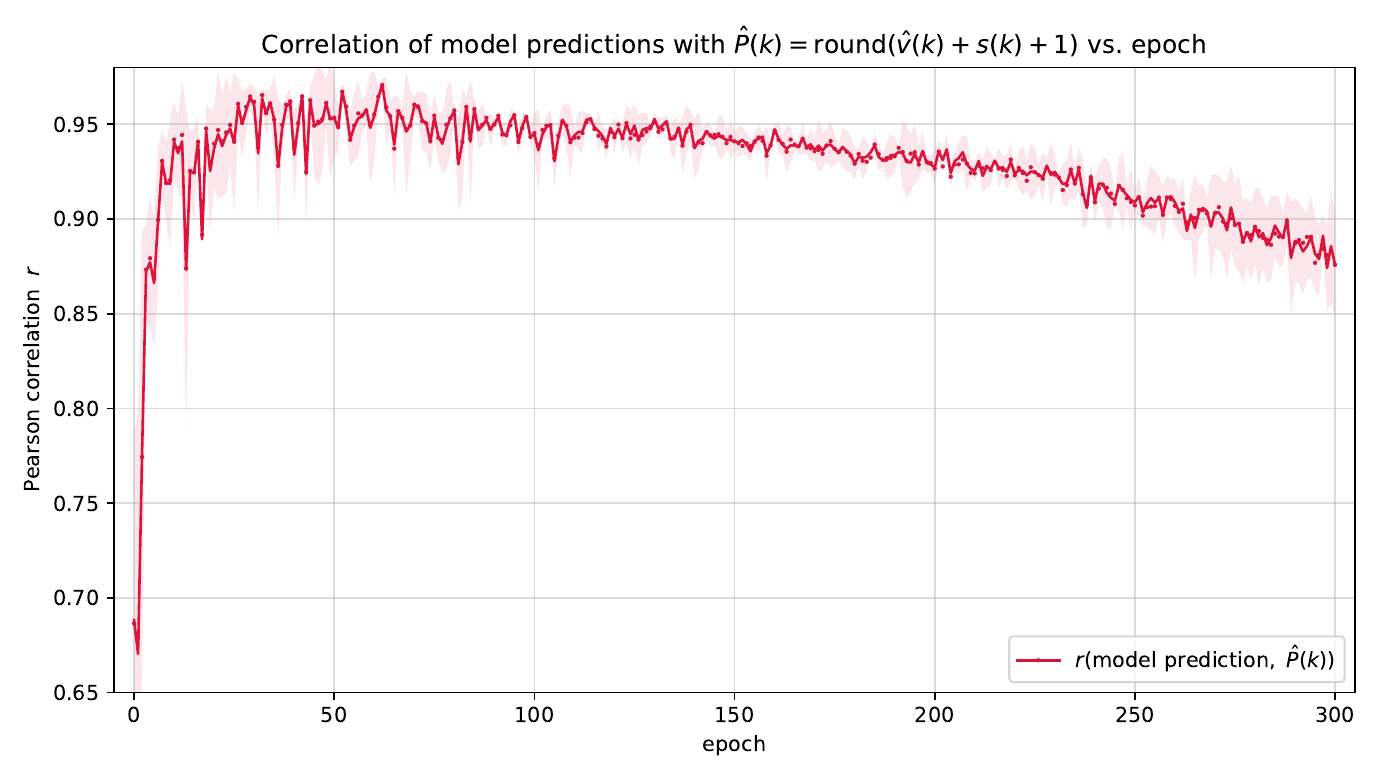}
    \caption{The Pearson correlation between the model's greedy
    predictions and our predictor $\hat{P}(k)$ on the validation
    set as a function of the epoch. The curve is the mean over
    three runs, and the shaded region shows one sample standard
    deviation on either side of the mean.}
    \label{fig:pearson_pred}
\end{figure}

On the validation set, the predictor $\hat{P}(k)$
is correct for $719$ of the
$3{,}000$ examples, giving an accuracy of
$23.97\%$.
The model's accuracy on the first plateau is
approximately $23\%$--$25\%$.
Together with the confusion matrices, these similar
accuracies support the interpretation that the
model's predictions approximate $\hat{P}(k)$ during
the first plateau.

\subsubsection{Learning around the second plateau of Figure~\ref{fig:v2_val_acc}} \label{sec:secondplateau}

The plateau diagnostics below use the validation predictions
from the three runs.
For the second plateau, it becomes much less clear
\emph{what} the model actually learned.
To investigate this, we consider the proportions
\[
q_g:=\mathbb{P}\bigl(\text{model is correct}\mid
G(k)=v_2(n_k)-\hat{v}(k)=g\bigr)
\]
on the validation set.
When $g=0$, the lower bound in
Conjecture~\ref{conj:low_bound} is sharp.
The mean $q_0$ across the three runs increases from $23.13\%$
at epoch~$40$ to $61.67\%$ at the validation-selected checkpoints
(epochs $261$, $272$, and $211$, in run order).
Figure~\ref{fig:correct_given_G_vs_epoch} shows how $q_0$ changes
over training.\begin{figure}[h]
    \centering
    \includegraphics[width=0.65\linewidth]{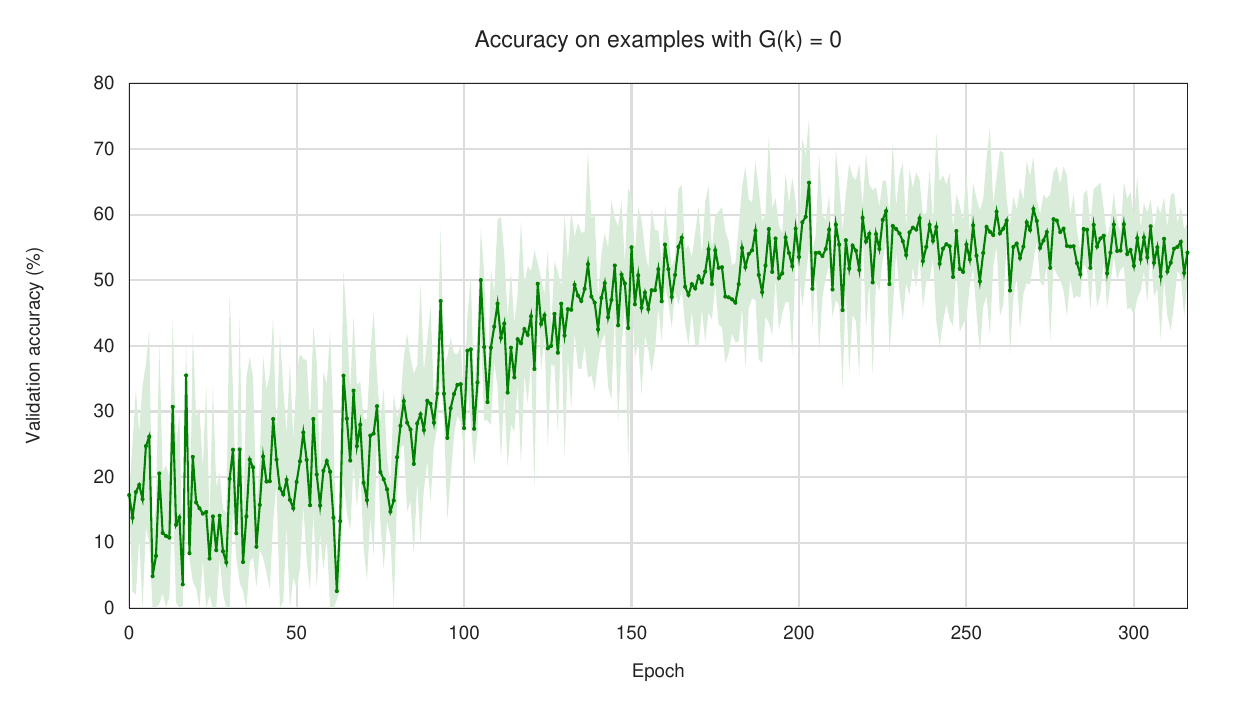}
    \caption{The percentage of correct greedy predictions
    among the $647$ validation examples satisfying
    $v_2(n_k)=\hat{v}(k)$, as a function of the epoch.
    The curve shows the mean across the three runs, and the
    shaded region shows one sample standard deviation on either
    side of the mean. Epochs $0$ through $316$ are shown.}
    \label{fig:correct_given_G_vs_epoch}
\end{figure}

From epoch~$40$ to the validation-selected checkpoints, the mean
$q_1$ increases from $37.37\%$ to $68.85\%$, while the mean
$q_2$ decreases from $30.73\%$ to $24.50\%$.
The largest contributions to the increase in overall validation
accuracy come from $g=0$ and $g=1$, with about $249$ and $226$
additional correct predictions per run on average, respectively.
See Table~\ref{tab:correct_by_G}.

\begin{table}[ht]
  \centering
  \begin{tabular}{lrrr}
    \toprule
    $G(k)$ & Count & Epoch $40$ (\%) & Selected checkpoint (\%) \\
    \midrule
    $0$     & $647$ & $23.13\pm15.38$ & $61.67\pm11.49$ \\
    $1$     & $718$ & $37.37\pm16.85$ & $68.85\pm2.37$  \\
    $2$     & $679$ & $30.73\pm22.82$ & $24.50\pm4.27$  \\
    $3$     & $411$ & $10.46\pm8.43$  & $4.79\pm2.44$   \\
    $4$     & $268$ & $0.50\pm0.43$   & $2.49\pm0.78$   \\
    $\geq5$ & $277$ & $0.00\pm0.00$   & $1.20\pm0.42$   \\
    \bottomrule
  \end{tabular}
  \caption{Exact-match validation accuracy by excess $G(k)$.
    Both accuracy columns use the same $3{,}000$ validation examples.
    Values are means and sample standard deviations across three runs.
    The validation-selected checkpoints are at epochs $261$, $272$,
    and $211$, in run order.}
  \label{tab:correct_by_G}
\end{table}

To test whether the improvement over $\hat P(k)$ persists
on contiguous held-out intervals, we repeated the $2$-adic
experiment with inputs $k$ and $\operatorname{factor}(\Delta_k)$
encoded in base $2$, using ten random seeds for each of two splits.
In the first split, we trained on
$1\leq k\leq244{,}000$ ($k\ne4$), used
$244{,}001\leq k\leq247{,}000$ for validation, and held out
$247{,}001\leq k\leq250{,}000$ for testing. In the second split,
we trained on
$[1,125{,}000]\cup[131{,}001,250{,}000]$, excluding $k=4$,
used $125{,}001\leq k\leq128{,}000$ for validation, and held out
$128{,}001\leq k\leq131{,}000$ for testing.
In each run, we selected the earliest epoch attaining the highest
validation greedy exact-match accuracy among epochs $0,\ldots,299$,
and report the corresponding archived test predictions.
On the upper test block, mean exact-match accuracy rose from
$23.38\%$ at epoch~$40$ to $(37.61\pm0.91)\%$ at the
selected checkpoints, compared with $24.83\%$ for $\hat P(k)$.
On the interior test block, it rose from $23.45\%$ to
$(37.86\pm1.34)\%$, compared with $24.03\%$ for $\hat P(k)$.
All comparisons within a split use the same $3{,}000$ test examples. At the selected checkpoints, every run of both
splits has higher exact-match accuracy than $\hat P(k)$.
Thus the improvement persists on both contiguous held-out intervals.

 The learning observed with $\text{factor}(\Delta_k)$
as a feature can be interpreted in light of
equations \eqref{eq:hatvdef} and \eqref{eq:v2approx}.
The prime factors and their exponents in the
discriminant provide information about the
contributions from $k$ and $k^2-16$, whose odd
prime divisors have different weights in
$\hat{v}(k)$.
Among the archived runs, all of which include $k$ as an input,
the highest greedy validation accuracies over all $500$ epochs are
$37.63\%$ with $\mathrm{factor}(\Delta_k)$, $26.87\%$ with
$\mathrm{factor}(N_k)$, and $36.47\%$ with both factorizations.
The highest value with $\mathrm{factor}(\Delta_k)$ and $c_2$
is $38.00\%$. Across runs with input pairs $(k,c_2)$, $(k,N_k)$,
and $(k,\Delta_k)$, peak validation accuracies range from
$16.03\%$ to $17.27\%$.

\section{Conclusions and future directions}
\label{sec:conclusions}

In this work, we have investigated the rational factors $r_k$ appearing
conjecturally in Boyd's identities \eqref{eq:conjecture}
by combining a large-scale numerical study with machine-learning experiments.
Rather than revealing a single formula for $r_k$, our explorations reveal
several layers of arithmetic structure. The sign of $r_k$ is governed by the
root number, while the conductor determines the order of magnitude of
$n_k=1/r_k$.
At the same time, the divisibility properties of $n_k$ are strongly influenced
by the local arithmetic of $E_k$.

One of the most striking conclusions of the statistical
analysis is the difference between the behavior at
different primes. For primes $p\geq5$, the observed
positive valuation frequencies are close to $p^{-r}$
for $r\geq1$. At $p=5$, the agreement improves
substantially when we restrict to curves with good
reduction at $5$.

The primes $2$ and $3$ exhibit richer arithmetic.
At $3$, congruence conditions on $k$ interact with
the parity of the number of bad primes congruent
to $1$ modulo $3$, producing observed parity
constraints on large subfamilies.
At $2$, a weighted count of the odd primes dividing
$k$ and $k^2-16$, together with a correction arising
from the reduction at $2$, leads to the lower bound
in Conjecture~\ref{conj:low_bound}.
The empirical mean of the excess over this bound
depends on the $2$-adic class of $k$ and is
approximately described by $2+s(k)$.
For the overall valuation histogram, we use the
reference model in \eqref{eq:v2approx}, which combines
$\hat v(k)+\mathrm{round}(s(k))$ with an auxiliary
negative-binomial variable. This gives a first-order
description of the observed distribution.

The machine-learning experiments clarify both the possibilities and the
limitations of the present approach. Predicting $r_k$ exactly proved to be too
difficult, but the model learned its size very effectively when the conductor, or factorizations from which it can be recovered, were included among the inputs. The experiments
therefore confirm that the choice and representation of arithmetic features
are crucial: the factorizations of the conductor and discriminant are much
more informative for the valuation problems than the corresponding unfactored
integers. For $p=5$ and $p=7$, the frequent zero class provides a strong
baseline, illustrating the difficulty caused by highly unbalanced
valuation distributions. Some runs at $p=5$ correctly predict positive
valuations, but their overall accuracy at the final epoch is below
that of the constant-zero predictor.

The results at $2$ and $3$ are more encouraging. For the $3$-adic valuation,
the model recognizes several of the congruence and parity phenomena found in
the statistical analysis. It also exceeds the constant-prediction baseline
on the tested examples with $v_3(n_k)\in\{0,1\}$ outside
$C''\cup\{k\equiv0,\pm5\pmod{27}\}$. More surprisingly,
in the $2$-adic experiment the evolution of the model appears to have a
mathematical interpretation. During the first plateau, its predictions are
very close to the deterministic predictor $\widehat P(k)$ derived from our
statistical model. During the second plateau, however, its accuracy increases
while its correlation with $\widehat P(k)$ decreases.
Across the three $2$-adic runs, the largest contributions to the
increase in validation accuracy from epoch~$40$ to the
validation-selected checkpoints come from examples with $G(k)=0$
and $G(k)=1$.
The three validation-selected checkpoints achieve a mean
test exact-match accuracy of $(37.12\pm1.60)\%$, compared with
$24.57\%$ for $\widehat P(k)$ on the same $3{,}000$ test examples.
This behavior suggests predictive relationships within the computed
range that are not fully captured by the explicit predictor
$\widehat P(k)$. The improvement also persists in the two blocked-split
experiments. One tests prediction beyond the largest training value,
while the other tests a contiguous omitted interval with training
values on both sides. These results do not establish accuracy for
substantially larger values of $k$ or identify a general arithmetic rule.

These observations lead to several natural questions.
 The most important
is to find a conceptual explanation for the arithmetic of $n_k$. This includes
explaining its near-integrality, proving or refining the proposed distribution
for $v_p(n_k)$ when $p\geq5$, and understanding the congruence and parity laws
at $3$.  At the prime $2$, one would like to prove
Conjecture~\ref{conj:low_bound}, characterize the
values of $k$ for which equality holds, and
understand the distribution of the excess $G(k)$
and its dependence on the $2$-adic class of $k$. The conditional formula discussed in
Remark~\ref{rem:general-BK} suggests a possible framework for such
an explanation.  It would be interesting to develop arithmetic
statistics for the families of Tate-twisted Bloch--Kato Selmer
groups and regulator indices occurring in that formula.  Existing
probabilistic models for ordinary Selmer groups, such as those
surveyed by Poonen \cite{Poonen}, may provide useful guidance, although the groups
appearing here are different and the regulator-index term must also
be modeled.
For $p\geq5$, the reference model assigns probability
$p^{-r}$ to valuation $r\geq1$. Hence, in this model,
the probability of a positive valuation is
\[
\sum_{r=1}^{\infty}p^{-r}=\frac{1}{p-1}.
\]
Conditional on a positive valuation, the distribution
is the usual geometric distribution.
The main arithmetic question concerns this increased
probability of the first factor of $p$, compared with
$1/p$ in a uniform-distribution model modulo $p$.

It would be interesting to determine whether this
enhancement arises from the local Tamagawa factors, the
Bloch--Kato Selmer group, the regulator index, or an interaction
among these terms.  Explicit approaches to $p$-adic regulators,
such as those developed by Coleman--de Shalit~\cite{Coleman-de-Shalit}, Besser ~\cite{Besser}, and
Asakura--Chida ~\cite{AsakuraChida}, may be useful in investigating the regulator
contribution. Ultimately, these phenomena should be related to the regulator
interpretations underlying Boyd's conjectures and to the local arithmetic of
the elliptic curves $E_k$.

The machine-learning results also suggest  directions for future
experiments.
Instead of predicting $v_2(n_k)$ directly, one could train a model
to predict the excess $G(k)=v_2(n_k)-\widehat v(k),$
or simply to classify the cases in which $G(k)=0$. We attempted to use the model to identify the cases in which $G(k)=0$, but did not obtain a useful classification. Interpreting such a
classifier, testing additional arithmetic features, and searching for
subfamilies on which it performs unusually well may lead to a refinement of
the present conjectures. For $p=5$ and $p=7$, $v_p(n_k)=0$ is a strong baseline.
This constant prediction achieves high overall
accuracy because zero is the most frequent observed
class, while failing to identify positive valuations. A class-balanced loss, which assigns greater weight to the rarer positive valuations, or a probabilistic model predicting the full distribution of $v_p(n_k)$, may help improve prediction of the rarer positive valuations.
A first natural extension, taking $k=\sqrt m$ for
$m\in\mathbb Z$ of either sign, is examined in
Section~\ref{sec:sqrtm}. In the computed range,
the numerically recognized rational candidates
$r_{\sqrt m}$ again have near-integral reciprocals.
An analogue of the $2$-adic lower bound persists,
and the odd-prime frequencies remain broadly similar.
For $p=5$, however, the excess of positive valuations
persists after restricting to good reduction.
The finer $3$-adic parity law of
Section~\ref{ssec:partitylaw} breaks down.
Finding its correct generalization to nonsquare
$m$ would be very interesting, as would testing
whether these patterns persist in other
one-parameter families of Mahler measures
associated with elliptic curves.

Thus, although the experiments do not provide a formula for $r_k$, they show
that these rational factors are far from arithmetically unstructured. The
combination of large-scale computation, statistical analysis, and machine
learning has produced concrete conjectures and has identified arithmetic phenomena that are not evident
from the currently known identities alone. The next step is
to turn these empirical regularities into arithmetic explanations and,
eventually, proofs.

\clearpage
\section{Appendix: Tables and Figures}
\label{sec:appendix}

\begin{table}[htbp]
\centering
\scriptsize
\begin{tabular}{@{}c|cccccccccc||c@{}}
$(p,r)$ & $[0,25K]$ & $[0,50K]$ & $[0,75K]$ & $[0,100K]$ & $[0,125K]$ & $[0,150K]$ & $[0,175K]$ & $[0,200K]$ & $[0,225K]$ & $[0,250K]$ & $p^{-r}$ \\
\hline
$(5,1)$ & 0.209 & 0.208 & 0.205 & 0.205 & 0.204 & 0.203 & 0.202 & 0.202 & 0.202 & 0.202 & 0.200 \\
$(5,2)$ & 0.044 & 0.044 & 0.044 & 0.044 & 0.043 & 0.043 & 0.043 & 0.043 & 0.043 & 0.043 & 0.040 \\
$(5,3)$ & 0.009 & 0.009 & 0.009 & 0.009 & 0.009 & 0.009 & 0.008 & 0.008 & 0.009 & 0.009 & 0.008 \\
$(5,4)$ & 0.002 & 0.002 & 0.002 & 0.002 & 0.002 & 0.002 & 0.002 & 0.002 & 0.002 & 0.002 & 0.002 \\
\hline\hline
$(7,1)$ & 0.142 & 0.143 & 0.144 & 0.144 & 0.143 & 0.142 & 0.142 & 0.142 & 0.142 & 0.142 & 0.143 \\
$(7,2)$ & 0.021 & 0.021 & 0.020 & 0.020 & 0.020 & 0.020 & 0.021 & 0.020 & 0.021 & 0.021 & 0.020 \\
$(7,3)$ & 0.003 & 0.003 & 0.003 & 0.003 & 0.003 & 0.003 & 0.003 & 0.003 & 0.003 & 0.003 & 0.003 \\
\hline\hline
$(11,1)$ & 0.088 & 0.090 & 0.091 & 0.090 & 0.091 & 0.091 & 0.091 & 0.090 & 0.091 & 0.091 & 0.091 \\
$(11,2)$ & 0.008 & 0.009 & 0.008 & 0.008 & 0.008 & 0.008 & 0.008 & 0.008 & 0.009 & 0.009 & 0.008 \\
$(11,3)$ & 0.001 & 0.001 & 0.001 & 0.001 & 0.001 & 0.001 & 0.001 & 0.001 & 0.001 & 0.001 & 0.001 \\
\hline\hline
$(13,1)$ & 0.074 & 0.075 & 0.076 & 0.076 & 0.076 & 0.076 & 0.076 & 0.075 & 0.076 & 0.076 & 0.077 \\
$(13,2)$ & 0.006 & 0.006 & 0.006 & 0.006 & 0.006 & 0.006 & 0.006 & 0.006 & 0.006 & 0.006 & 0.006 \\
$(13,3)$ & 0.000 & 0.001 & 0.001 & 0.000 & 0.000 & 0.000 & 0.000 & 0.000 & 0.000 & 0.000 & 0.000 \\
\hline\hline
$(17,1)$ & 0.059 & 0.058 & 0.059 & 0.059 & 0.059 & 0.059 & 0.059 & 0.059 & 0.059 & 0.059 & 0.059 \\
$(17,2)$ & 0.003 & 0.003 & 0.003 & 0.003 & 0.003 & 0.003 & 0.003 & 0.003 & 0.003 & 0.003 & 0.003 \\
\hline\hline
$(19,1)$ & 0.051 & 0.052 & 0.051 & 0.051 & 0.052 & 0.051 & 0.052 & 0.052 & 0.052 & 0.052 & 0.053 \\
$(19,2)$ & 0.003 & 0.003 & 0.003 & 0.003 & 0.003 & 0.003 & 0.003 & 0.003 & 0.003 & 0.003 & 0.003 \\
\hline\hline
$(23,1)$ & 0.043 & 0.043 & 0.043 & 0.043 & 0.043 & 0.043 & 0.043 & 0.043 & 0.043 & 0.043 & 0.043 \\
$(23,2)$ & 0.001 & 0.002 & 0.002 & 0.002 & 0.002 & 0.002 & 0.002 & 0.002 & 0.002 & 0.002 & 0.002 \\
\hline\hline
$(29,1)$ & 0.036 & 0.035 & 0.035 & 0.035 & 0.035 & 0.035 & 0.035 & 0.035 & 0.035 & 0.035 & 0.034 \\
\hline\hline
$(31,1)$ & 0.033 & 0.033 & 0.033 & 0.033 & 0.033 & 0.033 & 0.033 & 0.033 & 0.033 & 0.033 & 0.032 \\
\hline\hline
$(37,1)$ & 0.025 & 0.026 & 0.026 & 0.026 & 0.027 & 0.027 & 0.027 & 0.027 & 0.027 & 0.027 & 0.027 \\
\end{tabular}
\caption{Observed cumulative proportions of $k$ satisfying $v_p(n_k)=r$ for the pairs $(p,r)$ listed in the first column, across closed intervals of the form $[0,x]$, where $K$ denotes $1000$, together with the reference value $p^{-r}$.}
\label{tab:v5_v7_v11_v13_v17_v19_v23_v29_v31_v37_statistics_cumulative}
\end{table}

\begin{table}
\centering
\scriptsize
\begin{tabular}{@{}c|cccccccccc||c@{}}
$(p,r)$
& $[0,25K]$ & $[0,50K]$ & $[0,75K]$ & $[0,100K]$ & $[0,125K]$
& $[0,150K]$ & $[0,175K]$ & $[0,200K]$ & $[0,225K]$ & $[0,250K]$
& $p^{-r}$ \\
\hline
$(3,1)$ & 0.325 & 0.322 & 0.323 & 0.322 & 0.322 & 0.321 & 0.322 & 0.322 & 0.321 & 0.322 & 0.333 \\
$(3,2)$ & 0.231 & 0.233 & 0.234 & 0.234 & 0.234 & 0.234 & 0.234 & 0.234 & 0.235 & 0.234 & 0.111 \\
$(3,3)$ & 0.103 & 0.106 & 0.107 & 0.108 & 0.108 & 0.110 & 0.110 & 0.110 & 0.110 & 0.110 & 0.037 \\
$(3,4)$ & 0.037 & 0.041 & 0.042 & 0.042 & 0.042 & 0.042 & 0.042 & 0.042 & 0.043 & 0.043 & 0.012 \\
$(3,5)$ & 0.014 & 0.014 & 0.014 & 0.014 & 0.015 & 0.015 & 0.015 & 0.015 & 0.015 & 0.015 & 0.004 \\
$(3,6)$ & 0.005 & 0.005 & 0.005 & 0.005 & 0.005 & 0.005 & 0.005 & 0.005 & 0.005 & 0.005 & 0.001 \\
\end{tabular}
\caption{Observed cumulative proportions of $k$ satisfying $v_3(n_k)=r$ for the pairs $(3,r)$ listed in the first column, across closed intervals of the form $[0,x]$, where $K$ denotes $1000$.}
\label{tab:v3_statistics_cumulative}
\end{table}

\begin{table}[htbp]
\centering
\scriptsize
\begin{tabular}{@{}cc|cccccccccc@{}}
$(p,r)$ & $k\pmod 3$ & $[0,25K]$ & $[0,50K]$ & $[0,75K]$ & $[0,100K]$ & $[0,125K]$ & $[0,150K]$ & $[0,175K]$ & $[0,200K]$ & $[0,225K]$ & $[0,250K]$ \\
\hline
$(3,1)$ & all & 0.718 & 0.724 & 0.728 & 0.728 & 0.728 & 0.730 & 0.731 & 0.731 & 0.732 & 0.732 \\
 & $0$ & 0.754 & 0.755 & 0.757 & 0.757 & 0.757 & 0.758 & 0.758 & 0.758 & 0.758 & 0.758 \\
 & $1$ & 0.698 & 0.708 & 0.714 & 0.713 & 0.713 & 0.714 & 0.715 & 0.717 & 0.718 & 0.719 \\
 & $2$ & 0.700 & 0.708 & 0.713 & 0.715 & 0.716 & 0.717 & 0.718 & 0.719 & 0.719 & 0.720 \\
\hline
$(3,2)$ & all & 0.392 & 0.401 & 0.405 & 0.406 & 0.407 & 0.408 & 0.409 & 0.410 & 0.411 & 0.411 \\
 & $0$ & 0.432 & 0.436 & 0.440 & 0.442 & 0.442 & 0.444 & 0.444 & 0.445 & 0.444 & 0.445 \\
 & $1$ & 0.375 & 0.385 & 0.389 & 0.389 & 0.389 & 0.390 & 0.391 & 0.393 & 0.394 & 0.394 \\
 & $2$ & 0.370 & 0.382 & 0.384 & 0.388 & 0.389 & 0.391 & 0.391 & 0.392 & 0.394 & 0.393 \\
\hline
$(3,3)$ & all & 0.161 & 0.169 & 0.171 & 0.172 & 0.173 & 0.175 & 0.175 & 0.176 & 0.176 & 0.176 \\
 & $0$ & 0.187 & 0.193 & 0.194 & 0.196 & 0.195 & 0.196 & 0.197 & 0.197 & 0.198 & 0.198 \\
 & $1$ & 0.149 & 0.156 & 0.161 & 0.163 & 0.163 & 0.165 & 0.166 & 0.166 & 0.166 & 0.166 \\
 & $2$ & 0.148 & 0.156 & 0.157 & 0.159 & 0.161 & 0.163 & 0.163 & 0.164 & 0.165 & 0.165 \\
\hline
\end{tabular}
\caption{Observed cumulative proportions of $k$ satisfying $v_3(n_k)\ge r$ for the pairs $(3,r)$ listed in the first column, split according to congruence classes of $k \pmod 3$, across closed intervals of the form $[0,x]$, where $K$ denotes $1000$. }
\label{tab:v3_geq_mod3_statistics_congruence_cumulative}
\end{table}

\begin{figure}[htbp]
    \centering
    \includegraphics[scale=0.7]{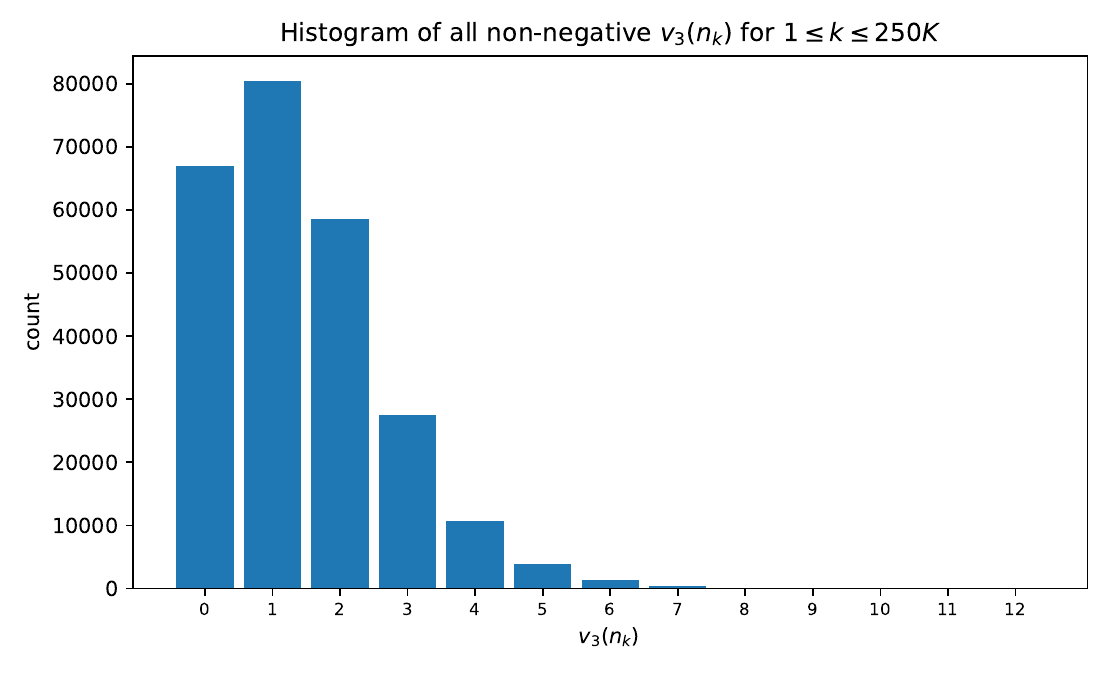}

    \caption{Histogram of the $3$-adic valuations $v_3(n_k)$ for $1 \leq k \leq 250K$ such that  $v_3(n_k) \geq 0$, where $K=1000$.}
    \label{fig:histogram_v3_0_250K}
\end{figure}

\begin{table}[htbp]
\centering
\scriptsize
\begin{tabular}{@{}c|cccccccccc@{}}
$(p,r)$ & $[0,25K]$ & $[0,50K]$ & $[0,75K]$ & $[0,100K]$ & $[0,125K]$ & $[0,150K]$ & $[0,175K]$ & $[0,200K]$ & $[0,225K]$ & $[0,250K]$ \\
\hline
$(2,1)$ & 0.000 & 0.000 & 0.000 & 0.000 & 0.000 & 0.000 & 0.000 & 0.000 & 0.000 & 0.000 \\
$(2,2)$ & 0.002 & 0.002 & 0.001 & 0.001 & 0.001 & 0.001 & 0.001 & 0.001 & 0.001 & 0.001 \\
$(2,3)$ & 0.011 & 0.008 & 0.007 & 0.006 & 0.005 & 0.005 & 0.005 & 0.005 & 0.004 & 0.004 \\
$(2,4)$ & 0.028 & 0.021 & 0.018 & 0.016 & 0.015 & 0.014 & 0.014 & 0.013 & 0.012 & 0.012 \\
$(2,5)$ & 0.053 & 0.043 & 0.038 & 0.035 & 0.033 & 0.031 & 0.030 & 0.029 & 0.028 & 0.027 \\
$(2,6)$ & 0.089 & 0.075 & 0.068 & 0.064 & 0.061 & 0.058 & 0.056 & 0.055 & 0.053 & 0.052 \\
$(2,7)$ & 0.131 & 0.115 & 0.107 & 0.102 & 0.097 & 0.093 & 0.091 & 0.089 & 0.087 & 0.085 \\
$(2,8)$ & 0.147 & 0.140 & 0.133 & 0.129 & 0.126 & 0.123 & 0.121 & 0.119 & 0.117 & 0.115 \\
$(2,9)$ & 0.148 & 0.147 & 0.146 & 0.144 & 0.143 & 0.141 & 0.140 & 0.138 & 0.137 & 0.136 \\
$(2,10)$ & 0.125 & 0.134 & 0.137 & 0.137 & 0.137 & 0.138 & 0.138 & 0.138 & 0.138 & 0.138 \\
$(2,11)$ & 0.095 & 0.108 & 0.114 & 0.118 & 0.121 & 0.123 & 0.124 & 0.125 & 0.125 & 0.126 \\
$(2,12)$ & 0.067 & 0.080 & 0.086 & 0.090 & 0.093 & 0.096 & 0.098 & 0.100 & 0.102 & 0.103 \\
$(2,13)$ & 0.045 & 0.053 & 0.058 & 0.062 & 0.065 & 0.068 & 0.070 & 0.072 & 0.074 & 0.075 \\
$(2,14)$ & 0.026 & 0.033 & 0.037 & 0.040 & 0.042 & 0.044 & 0.045 & 0.047 & 0.048 & 0.050 \\
$(2,15)$ & 0.015 & 0.019 & 0.021 & 0.024 & 0.026 & 0.028 & 0.029 & 0.031 & 0.031 & 0.032 \\
$(2,16)$ & 0.007 & 0.010 & 0.013 & 0.014 & 0.015 & 0.016 & 0.017 & 0.018 & 0.018 & 0.019 \\
$(2,17)$ & 0.004 & 0.006 & 0.007 & 0.008 & 0.009 & 0.009 & 0.010 & 0.010 & 0.011 & 0.011 \\
$(2,18)$ & 0.003 & 0.003 & 0.004 & 0.005 & 0.005 & 0.005 & 0.005 & 0.006 & 0.006 & 0.006 \\
$(2,19)$ & 0.002 & 0.002 & 0.002 & 0.002 & 0.003 & 0.003 & 0.003 & 0.003 & 0.003 & 0.003 \\
$(2,20)$ & 0.001 & 0.001 & 0.001 & 0.001 & 0.001 & 0.001 & 0.002 & 0.002 & 0.002 & 0.002 \\
\end{tabular}
\caption{Observed cumulative proportions of $k$ satisfying  $v_2(n_k)=r$ for the pairs $(2,r)$ listed in the first column, across closed intervals of the form $[0,x]$, where $K$ denotes $1000$.}
\label{tab:v2_statistics_cumulative}
\end{table}

\begin{table}[htbp]
\centering
\scriptsize
\begin{tabular}{@{}cc|cccccccccc@{}}
$(p,r)$ & $k\pmod 7$ & $[0,25K]$ & $[0,50K]$ & $[0,75K]$ & $[0,100K]$ & $[0,125K]$ & $[0,150K]$ & $[0,175K]$ & $[0,200K]$ & $[0,225K]$ & $[0,250K]$ \\
\hline
$(2,6)$ & all & 0.905 & 0.926 & 0.935 & 0.941 & 0.945 & 0.948 & 0.951 & 0.953 & 0.954 & 0.956 \\
 & $0$ & 0.921 & 0.938 & 0.947 & 0.952 & 0.956 & 0.958 & 0.961 & 0.963 & 0.964 & 0.965 \\
 & $1$ & 0.863 & 0.892 & 0.905 & 0.914 & 0.919 & 0.923 & 0.927 & 0.930 & 0.933 & 0.935 \\
 & $2$ & 0.870 & 0.896 & 0.908 & 0.914 & 0.919 & 0.923 & 0.927 & 0.930 & 0.932 & 0.934 \\
 & $3$ & 0.971 & 0.982 & 0.987 & 0.989 & 0.990 & 0.991 & 0.992 & 0.992 & 0.993 & 0.993 \\
 & $4$ & 0.972 & 0.981 & 0.985 & 0.988 & 0.989 & 0.990 & 0.991 & 0.991 & 0.992 & 0.992 \\
 & $5$ & 0.865 & 0.896 & 0.909 & 0.916 & 0.921 & 0.926 & 0.928 & 0.931 & 0.933 & 0.935 \\
 & $6$ & 0.872 & 0.896 & 0.908 & 0.916 & 0.922 & 0.925 & 0.928 & 0.931 & 0.933 & 0.935 \\
\hline
$(2,7)$ & all & 0.816 & 0.851 & 0.867 & 0.877 & 0.884 & 0.890 & 0.894 & 0.898 & 0.901 & 0.903 \\
 & $0$ & 0.835 & 0.870 & 0.888 & 0.896 & 0.902 & 0.907 & 0.911 & 0.914 & 0.917 & 0.919 \\
 & $1$ & 0.754 & 0.800 & 0.818 & 0.833 & 0.841 & 0.848 & 0.854 & 0.859 & 0.862 & 0.865 \\
 & $2$ & 0.767 & 0.805 & 0.823 & 0.833 & 0.841 & 0.848 & 0.853 & 0.859 & 0.863 & 0.866 \\
 & $3$ & 0.917 & 0.942 & 0.952 & 0.958 & 0.962 & 0.964 & 0.967 & 0.969 & 0.970 & 0.971 \\
 & $4$ & 0.925 & 0.944 & 0.954 & 0.959 & 0.962 & 0.965 & 0.967 & 0.969 & 0.970 & 0.972 \\
 & $5$ & 0.754 & 0.798 & 0.818 & 0.832 & 0.842 & 0.849 & 0.853 & 0.858 & 0.862 & 0.865 \\
 & $6$ & 0.758 & 0.799 & 0.817 & 0.831 & 0.841 & 0.847 & 0.853 & 0.857 & 0.862 & 0.864 \\
\hline
\end{tabular}
\caption{Observed cumulative proportions of $k$ satisfying $v_2(n_k)\ge r$ for the pairs $(2,r)$ listed in the first column, split according to congruence classes of $k \pmod 7$, across closed intervals of the form $[0,x]$, where $K$ denotes $1000$.
\label{tab:v2_geq_mod7_statistics_congruence_cumulative}}
\end{table}

\begin{figure}[htbp]
    \centering
    \includegraphics[scale=0.7]{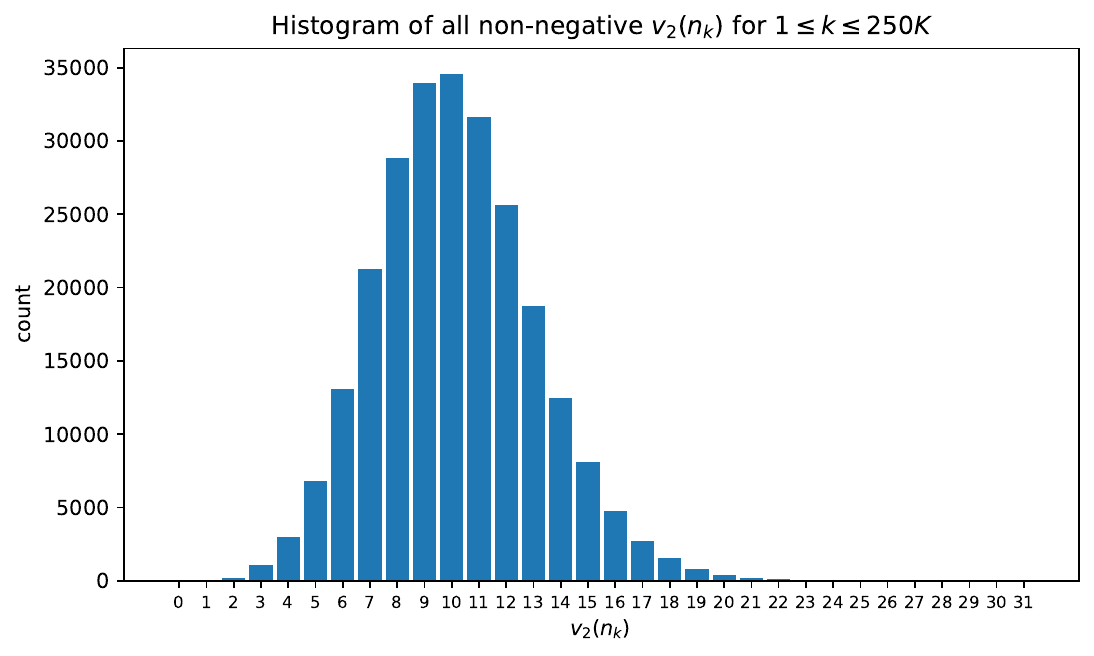}

    \caption{Histogram of the $2$-adic valuations $v_2(n_k)$ for $1 \leq k \leq 250K$ such that $v_2(n_k) \geq 0$, where $K=1000$.}
    \label{fig:histogram_v2_0_250K}
\end{figure}

\begin{table}[htbp]
\centering\small
\setlength{\tabcolsep}{4.5pt}
\begin{tabular}{r l l c c c c}
\toprule
$k$ & $k-4$ & $k+4$ & $W(k)$ & $\hat v(k)$ & $v_2(n_k)$ & $v_2(n_k)-\hat v(k)$\\
\midrule
$321559$  & $3\cdot5\cdot13\cdot17\cdot97$           & $11\cdot23\cdot31\cdot41$            & $21$ & $17$ & $19$ & $+2$\\
$400201$  & $3\cdot7\cdot17\cdot19\cdot59$           & $5\cdot13\cdot47\cdot131$           & $21$ & $17$ & $20$ & $+3$\\
$440891$  & $23\cdot29\cdot661$                      & $3\cdot5\cdot7\cdot13\cdot17\cdot19$ & $21$ & $17$ & $19$ & $+2$\\
$471709$  & $3\cdot5\cdot13\cdot41\cdot59$           & $11\cdot19\cdot37\cdot61$           & $21$ & $17$ & $21$ & $+4$\\
$474415$  & $3\cdot7\cdot19\cdot29\cdot41$           & $11\cdot17\cdot43\cdot59$           & $21$ & $17$ & $18$ & $+1$\\
$519589$  & $3\cdot5\cdot11\cdot47\cdot67$           & $19\cdot23\cdot29\cdot41$           & $21$ & $17$ & $20$ & $+3$\\
$572029$  & $3\cdot5^{2}\cdot29\cdot263$             & $7\cdot11\cdot17\cdot19\cdot23$     & $21$ & $17$ & $22$ & $+5$\\
$729491$  & $11\cdot17\cdot47\cdot83$                & $3^{2}\cdot5\cdot13\cdot29\cdot43$  & $21$ & $17$ & $21$ & $+4$\\
$889291$  & $3\cdot7\cdot17\cdot47\cdot53$           & $5\cdot11\cdot19\cdot23\cdot37$     & $23$ & $19$ & $21$ & $+2$\\
$1008011$ & $7\cdot11\cdot13\cdot19\cdot53$          & $3\cdot5\cdot17\cdot59\cdot67$      & $22$ & $18$ & $20$ & $+2$\\
$1394701$ & $3\cdot17\cdot23\cdot29\cdot41$          & $5\cdot13\cdot43\cdot499$           & $22$ & $18$ & $21$ & $+3$\\
$1601149$ & $3^{2}\cdot5\cdot7\cdot13\cdot17\cdot23$ & $83\cdot101\cdot191$                & $22$ & $18$ & $21$ & $+3$\\
$1847751$ & $11\cdot17\cdot41\cdot241$              & $5\cdot7\cdot13\cdot31\cdot131$         & $22$ & $18$ & $23$ & $+5$\\
$1944939$ & $5\cdot19\cdot59\cdot347$               & $7\cdot11\cdot13\cdot29\cdot67$         & $22$ & $18$ & $22$ & $+4$\\
\bottomrule
\end{tabular}
\caption{Prime-rich odd values of $k$ in $2.5\times10^{5}<k<2\times10^{6}$ chosen to
maximise $W(k):=\omega_{\mathrm{odd}}(k)+2\,\omega_{\mathrm{odd}}(k^2-16)$; the maximum
$W=23$ is attained uniquely at $k=889291$. For odd $k$ the curve $E_k$ is semistable at
$2$, so $c(k)=4$ and $\hat v(k)=W(k)-4$. The table includes every odd $k$ in the window
with $W\ge22$, together with a selection of those with $W=21$. In every case
$v_2(n_k)\ge\hat v(k)$. }
\label{tab:largeW-odd}
\end{table}

\begin{table}[htbp]
\centering\small
\setlength{\tabcolsep}{4.5pt}
\begin{tabular}{r l l c c c c}
\toprule
$k$ & $k-4$ & $k+4$ & $W(k)$ & $\hat v(k)$ & $v_2(n_k)$ & $v_2(n_k)-\hat v(k)$\\
\midrule
$940474$  & $2\cdot3\cdot5\cdot23\cdot29\cdot47$     & $2\cdot7\cdot11\cdot31\cdot197$         & $21$ & $19$ & $21$ & $+2$\\
$1069814$ & $2\cdot5\cdot7\cdot17\cdot29\cdot31$     & $2\cdot3\cdot37\cdot61\cdot79$          & $21$ & $19$ & $27$ & $+8$\\
$1195674$ & $2\cdot5\cdot7\cdot19\cdot29\cdot31$     & $2\cdot11\cdot17\cdot23\cdot139$        & $21$ & $19$ & $20$ & $+1$\\
$1204606$ & $2\cdot3\cdot7\cdot23\cdot29\cdot43$     & $2\cdot5\cdot11\cdot47\cdot233$         & $21$ & $19$ & $22$ & $+3$\\
$1234874$ & $2\cdot5\cdot7\cdot13\cdot23\cdot59$     & $2\cdot3\cdot29\cdot47\cdot151$         & $21$ & $19$ & $22$ & $+3$\\
$1258786$ & $2\cdot3\cdot7\cdot17\cdot41\cdot43$     & $2\cdot5\cdot13\cdot23\cdot421$         & $21$ & $19$ & $21$ & $+2$\\
$1286666$ & $2\cdot13\cdot17\cdot41\cdot71$          & $2\cdot3\cdot5\cdot7\cdot11\cdot557$    & $21$ & $19$ & $21$ & $+2$\\
$1540136$ & $2^{2}\cdot11\cdot17\cdot29\cdot71$      & $2^{2}\cdot3\cdot5\cdot7\cdot19\cdot193$ & $21$ & $19$ & $19$ & $+0$\\
$1587226$ & $2\cdot3^{3}\cdot7\cdot13\cdot17\cdot19$ & $2\cdot5\cdot23\cdot67\cdot103$         & $21$ & $19$ & $20$ & $+1$\\
$1614826$ & $2\cdot3\cdot11\cdot43\cdot569$          & $2\cdot5\cdot7\cdot17\cdot23\cdot59$    & $21$ & $19$ & $20$ & $+1$\\
$1653526$ & $2\cdot3\cdot13\cdot17\cdot29\cdot43$    & $2\cdot5\cdot37\cdot41\cdot109$         & $21$ & $19$ & $22$ & $+3$\\
$1669906$ & $2\cdot3\cdot13\cdot79\cdot271$          & $2\cdot5\cdot11\cdot17\cdot19\cdot47$   & $21$ & $19$ & $20$ & $+1$\\
$1700846$ & $2\cdot11\cdot13\cdot19\cdot313$         & $2\cdot3\cdot5^{2}\cdot17\cdot23\cdot29$ & $21$ & $19$ & $21$ & $+2$\\
$1703464$ & $2^{2}\cdot3\cdot5\cdot11\cdot29\cdot89$ & $2^{2}\cdot13\cdot17\cdot41\cdot47$     & $21$ & $19$ & $26$ & $+7$\\
$1717166$ & $2\cdot41\cdot43\cdot487$                & $2\cdot3\cdot5\cdot7\cdot13\cdot17\cdot37$ & $21$ & $19$ & $24$ & $+5$\\
$1722886$ & $2\cdot3\cdot7\cdot17\cdot19\cdot127$    & $2\cdot5\cdot13\cdot29\cdot457$         & $21$ & $19$ & $21$ & $+2$\\
$1762094$ & $2\cdot5\cdot11\cdot83\cdot193$          & $2\cdot3\cdot13\cdot19\cdot29\cdot41$   & $21$ & $19$ & $24$ & $+5$\\
$1774630$ & $2\cdot3\cdot7\cdot29\cdot31\cdot47$     & $2\cdot23\cdot173\cdot223$             & $21$ & $19$ & $22$ & $+3$\\
$1791566$ & $2\cdot17\cdot23\cdot29\cdot79$          & $2\cdot3\cdot5\cdot11\cdot61\cdot89$    & $21$ & $19$ & $24$ & $+5$\\
$1812794$ & $2\cdot5\cdot7\cdot19\cdot29\cdot47$     & $2\cdot3^{2}\cdot13\cdot61\cdot127$     & $21$ & $19$ & $19$ & $+0$\\
$1885594$ & $2\cdot3^{2}\cdot5\cdot7\cdot41\cdot73$  & $2\cdot11\cdot13\cdot19\cdot347$        & $21$ & $19$ & $21$ & $+2$\\
$1899206$ & $2\cdot19\cdot23\cdot41\cdot53$          & $2\cdot3\cdot5\cdot29\cdot37\cdot59$    & $21$ & $19$ & $21$ & $+2$\\
$1952786$ & $2\cdot13\cdot19\cdot59\cdot67$          & $2\cdot3\cdot5\cdot7\cdot17\cdot547$    & $21$ & $19$ & $20$ & $+1$\\
$1984814$ & $2\cdot5\cdot41\cdot47\cdot103$          & $2\cdot3\cdot11\cdot17\cdot29\cdot61$   & $21$ & $19$ & $19$ & $+0$\\
\bottomrule
\end{tabular}
\caption{Prime-rich even values of $k$ in $2.5\times10^{5}<k<2\times10^{6}$ chosen to
maximise $W(k):=\omega_{\mathrm{odd}}(k)+2\,\omega_{\mathrm{odd}}(k^2-16)$; for even $k$
the maximum is $W=21$. All are even with $16\nmid k$, so $E_k$ is additive at $2$,
$c(k)=2$ and $\hat v(k)=W(k)-2$. The table lists every even $k$ in the window attaining
$W=21$, i.e.\ the complete set of even maximisers. In every case $v_2(n_k)\ge\hat v(k)$,
with equality at $k=1540136$, $k=1812794$, and $k=1984814$.}
\label{tab:largeW-even}
\end{table}

\clearpage
\bibliographystyle{alpha}
\bibliography{Bibliography}

\end{document}